\documentclass[11pt, reqno]{amsart}
\usepackage{geometry}                
\usepackage[foot]{amsaddr}
\usepackage{layout}
\usepackage{graphicx}
\usepackage{epsfig}
\usepackage{booktabs}
\usepackage{amssymb}
\usepackage{epstopdf}
\usepackage{hyperref}
\usepackage{dsfont}
\usepackage{bbm}

\usepackage{caption}
\usepackage[most]{tcolorbox}
\usepackage{subcaption}
\usepackage{float}
\usepackage{algorithm}
\usepackage{algorithmic}
\makeatletter
\def\@setthanks{\vspace{-\baselineskip}\def\thanks##1{\@par##1\@addpunct.}\thankses}
\makeatother

\newcommand{\1}{\mathds{1}}

\newtheorem{theorem}{Theorem}
\newtheorem{lemma}[theorem]{Lemma}
\newtheorem{definition}[theorem]{Definition}

\newtheorem{proposition}[theorem]{Proposition}
\newtheorem{example}{Example}
\newtheorem{remark}[theorem]{Remark}

\title[The convergence rate of LQ-MFG]
{The convergence rate of the equilibrium measure for the hybrid LQG Mean Field Game}

\author{Jiamin Jian $^{\ast}$}
\email{jjian2@wpi.edu$^{\ast}$}

\author{Peiyao Lai $^{\dagger}$}
\email{plai@wpi.edu$^{\dagger}$}

\author{Qingshuo Song $^{\ddagger}$}
\email[Corresponding author]{qsong@wpi.edu$^{\ddagger}$}

\author{Jiaxuan Ye $^{\S}$}
\email{jye.wpi.edu$^{\S}$}

\thanks{Department of Mathematical Sciences, Worcester Polytechnic Institute}

\date{}                                           

\begin{document}
\maketitle
\begin{abstract}
In this work, we study the convergence rate of the $N$-player LQG game with a Markov chain common noise towards its asymptotic Mean Field Game. 
By postulating a Markovian structure via two auxiliary processes for the first and second moments of the Mean Field Game equilibrium and applying the fixed point condition in Mean Field Game, we first provide the characterization of the equilibrium measure in Mean Field Game with a finite-dimensional Riccati system of ODEs.  Additionally,  
with an explicit coupling of the optimal trajectory of the $N$-player game driven by $N$ dimensional Brownian motion and Mean Field Game counterpart driven by one-dimensional Brownian motion, we obtain the convergence rate $O(N^{-1/2})$ with respect to 2-Wasserstein distance.
\end{abstract}



\section{Introduction}
Mean Field Game (MFG) theory  is intended to describe an asymptotic limit of complex $N$-player differential game invariant to a reshuffling of the players' indices, and has attracted resurgent attention from numerous researchers in probability after its pioneering works of \cite[Lasry and Lions]{LJ07} and \cite[Huang, Caines, and Malhame]{HMC06}, and we refer to comprehensive descriptions to the book \cite[Carmona and Delarue]{carmona2018} and the references therein.
	
In this paper, we study the convergence rate of equilibrium measures of $N$-player differential game in the context of Linear-Quadratic (LQ) structure with a common noise to its limiting MFG system. Different from the works mentioned above, the common noise in this paper is a continuous-time Markov chain (CTMC) instead of Brownian motion, which often models the real-world control problems associated with hybrid systems. 
Markov chains are widely used to model systems that exhibit randomness and transition between different states. In various real-world scenarios, especially in economics (see \cite{TYL16}), finance (see \cite{ZY03}), biology (see \cite{ZY09}), and engineering (see \cite{ZYB97}), the dynamics of systems can be effectively represented as discrete states with probabilistic transitions between them. By using CTMC, the applications aim to model less frequently changing common noises, such as government policies implemented by two different regimes.

LQ control problems have been widely recognized in the stochastic control theory due to their broad applications. More importantly, LQ structure leads to solvability in a closed form, namely the Ricatti system, and this usually sheds light on many fundamental properties of the control theory. For this reason, 
LQ structure has also been studied in MFGs with or without common noises for its importance. 
The related literature include
major and minor LQG Mean Field Games system (\cite{Hua10, NSH12, FJC20});
social optimal in LQG Mean Field Games (\cite{HCM12, FHQ19}); the LQG Mean Field Games with different model settings (\cite{Bar12, HH13, BPF13, HLW15}); and LQG Graphon Mean Field Games (\cite{GCH20}). 
Recently, LQ Mean Field Games with a Brownian motion as the common noise have also been studied in  (\cite{ahuja2015mean, tchuendom2018}) with restrictions of the dependence of measure on its mean alone. Moreover, some literature considers various topics of  Mean Field control and game problems with Markov chain common noise, 
see \cite{LXZ22, NNY20, NYH20}.
	
A fundamental question in this regard is the convergence rate of $N$-player game to the desired MFG system. A well-known result is about the convergence rate of value functions of the generic player, which can be shown $O(N^{-1})$, see  for instance \cite{Car10, CDLL19, carmona2018, HY21}. In particular, \cite{HY21} establishes the convergence rate of value functions in the sense of 
$$J_1^N(\hat \alpha_1, \hat \alpha_{-1}) \le J_1^N(\alpha_1, \hat \alpha_{-1}) + O(N^{-1}), $$
where $J_1^N$ is the value of the first player in $N$-player game and $\hat \alpha$ is the Nash equilibrium decentralized control process for the Mean Field Game problem. 

In contrast, the convergence rate of equilibrium measures is another challenging question due to the complication of the correlation structures among $N$ players. 
To be more concrete, we examine the behavior of the $\hat X_{it}^{(N)}$, who represents the equilibrium state of the $i$-th player at time $t$ in the $N$-player game defined within the probability space $\left(\Omega^{(N)}, \mathcal F^{(N)}, \mathbb F^{(N)}, \mathbb P^{(N)} \right)$. Additionally, we denote $\hat X_t$ as the equilibrium path at time $t$ derived from the associated MFG defined in the probability space $(\Omega, \mathcal F, \mathbb F, \mathbb P)$.
The question pertains to the convergence of $\hat X_{1t}^{(N)}$ as follows:
\begin{itemize}
\item [(Q)]
The $\mathbb W_p$-convergence rate of the representative equilibrium path, 
$$\mathbb W_p \left(\mathcal L \left(\hat X_{1t}^{(N)} \right), \mathcal L \left(\hat X_t \right) \right) 
	= O \left(N^{-?} \right).$$
Here, $\mathbb W_p$ denotes the $p$-Wasserstein metric.
\end{itemize}

The existing literature extensively explores the convergence rate in this context. For (Q), Theorem 2.4.9 of the monograph \cite{CDLL19} establishes a convergence rate of $O(N^{-1/2})$ using the $\mathbb W_1$ metric. More recently, \cite{JT23} addresses (Q) by introducing displacement monotonicity and controlled common noise, and Theorem 2.23 applies the maximum principle of forward-backward propagation of chaos to achieve the same convergence rate. It is important to note that these results are not applicable to the Linear Quadratic Gaussian (LQG) framework, primarily due to the assumption concerning the linear growth of the cost functional.

The main result of this paper establishes that the equilibrium measures exhibit a convergence rate of $1/2$ concerning the 2-Wasserstein distance. The precise statement of this result can be found in Theorem \ref{t:main02}. In comparison to the aforementioned literature, two primary distinctions emerge. Firstly, within the framework of Mean Field Games, the common noise is modeled as a Continuous-Time Markov Chain. Secondly, a significant difference lies in the cost function's behavior, as it does not possess linear growth within the context of the Linear Quadratic Gaussian (LQG) framework.

To obtain the desired convergence rate in this paper, the first building block is the characterization of the equilibrium measure of the limiting MFG by a finite-dimensional ODE system. The key step leading us to a desired finite-dimensional system is that, instead of searching for the infinite-dimensional function directly, we postulate a Markovian structure via auxiliary processes \eqref{eq:mu_sigma} governed by its finite-dimensional coefficient functions, which exhibits the distinct feature of Markov chain common noise relatives to the Brownian motion counterpart.

The next stage towards the convergence rate is to compare the limiting MFG system to a $N$-player game. 
In contrast to the characterization of the MFG system, it is relatively routine to solve the $N$-player game due to its LQ structure.
Therefore, the convergence rate problem can be recasted to the following question about a coupling of the two following processes: 
For two equilibrium processes $\hat X$ of MFG in $\Omega$ and $\hat X_1^{(N)}$ of $N$-player game in $\Omega^{(N)}$, 
finding a random process $Z^N$ in $\Omega$ whose distribution is identical to $\hat X_1^{(N)}$ satisfying the estimate in the form of
$\mathbb E [|\hat X_t - Z^N_t |^2] = O \left(N^{-?} \right).$     
For this purpose, we first show an  $N$-invariant algebraic structure of the seemingly intractable $\kappa N^3$ dimensional ODE system \eqref{eq:ABC}, which originated from \cite[Huang and Yang]{HY21} as a dimensional reduction in the system with Brownian common noise. Thanks to this $N$-invariant structure, the complex ODE system \eqref{eq:ABC} can be reduced to the ODE system \eqref{eq:a_12} whose dimension agrees with the ODE \eqref{eq:ode1} of MFG system. Moreover,  $\hat X_1^{(N)}$ can be represented as a stochastic flow driven by two Brownian motions $W_1^{(N)}$ and $W_{-1}^{(N)}:=\frac 1 {\sqrt{N-1}} \sum_{i=2}^{N} W_i^{(N)}$, which enables us to embed the equilibrium process $\hat X_1^{(N)}$ to any probability space having only two Brownian motions. 
	
The rest of this paper is outlined as follows:
Section \ref{s:section2} presents a precise formulation of the problem and two main results. Section \ref{s:section3} is devoted to the derivation of our first result: the equilibrium of MFGs. In Section \ref{s:section4}, we show in detail the convergence of the $N$-player game to MFGs, which yields our second main result. Section \ref{s:section6} demonstrates the convergence by some numerical examples. 
The conclusion and some possible future works are summarised in Section \ref{s:conclusion}.
Section \ref{s:appendix} is an appendix that collects some related facts to support our main theme.


	\section{Problem setup and Main results}
	\label{s:section2}
	
	First, we collect common notations used in this paper in Subsection \ref{s:notation}.
	Then, we set up problems on
	MFGs and the $N$-player game separately in  Subsections \ref{s:mfg} and \ref{s:n-player}. 
The main results are presented in 	Subsection \ref{s:main} and some interpretations of our main results are added in Subsection \ref{s:remarks on main}. 
	
	\subsection{Notations} \label{s:notation}
		Let $T>0$ be a fixed terminal time and $(\Omega, \mathcal F_T, \mathbb F = \{\mathcal F_t: 0\le t\le T\}, \mathbb P)$ be a completed filtered probability space satisfying the usual conditions, on which
	$W$ and $B$ are two independent standard Brownian motions, and $Y$ is a continuous time Markov chain (CTMC) independent of $(W, B)$ taking values in a finite state space $\mathcal Y = \{1, 2, \dots, \kappa\}$ with a generator 
	\begin{equation}
	\label{eq:Y}
	    Q = (q_{i,j})_{i,j \in \mathcal Y}
	\end{equation} 
	satisfying $q_{i,j} \geq 0$ for all $i \neq j \in \mathcal Y$ and $\sum_{i \neq j} q_{i,j} + q_{i,i} = 0$ for each $i \in \mathcal Y$.
In the above, 
	the Brownian motion $B$ does not play any role in MFG problem formulation until the convergence proof of the $N$-player game to MFGs.

	By $L^p := L^p(\Omega, \mathbb P)$, we denote the space of random variables $X$ on 
	$(\Omega,  \mathcal F_T, \mathbb P)$ with finite $p$-th moment with norm $\|X\|_p = (\mathbb E \left[|X|^p \right])^{1/p}$. 
	We also denote by
	$L_{\mathbb F}^p:=L_{\mathbb F}^p([0, T] \times \Omega)$ the space of all $\mathbb F$-progressively measurable random processes $\alpha = (\alpha_t)_{0 \leq t \leq T}$ satisfying 
	$$\mathbb E \left[ \int_0^T |\alpha_t|^p dt \right] <\infty.$$
	
	For any polish (complete separable metric)  
	space $(P, \mathcal B(P), d)$, we use $\delta_x$ to denote the Dirac measure on the point $x\in P$.
	Then, the collection of all probabilities $m$ on $(P, \mathcal B(P), d)$
	 having finite $k$-th moment is denoted by $\mathcal P_k(P)$, i.e. 
	$$[m]_k := \int x^k m(dx) <\infty, \quad \forall m\in \mathcal P_k(P).$$

	The equilibrium of MFGs with the common noise yields the conditional distribution. For real-valued random variables 
	$X$ and $Z$ in $(\Omega, \mathcal F_T, \mathbb P)$, we denote the distribution of $X$ conditional on $\sigma(Z)$ by $\mathcal L(X|Z)$, or equivalently
	$$\mathcal L(X|Z) (A) = \mathbb E [ I_A(X) |Z ], \quad \forall A \in \mathcal F_T.$$
	Note that $\mathcal L(X|Z) (A): \Omega \mapsto \mathbb R$ is a $\sigma(Z)$-measurable random variable, 
	therefore, $\mathcal L(X|Z)$ is $\sigma(Z)$-measurable random probability distribution with $k$-th moment 
	$[\mathcal L(X|Z)]_k = \mathbb E[X^k|Z]$, if it exists. 
	We refer to more details on the conditional distribution in Volume II of \cite{carmona2018}.
	The next proposition provides an embedding approach to prove a convergence in distribution, 
	which will be used later in the convergence of the $N$-player game to MFGs.
	\begin{proposition}
		\label{p:conv}
		Suppose $(\Omega^{(N)}, \mathcal F_T^{(N)}, \mathbb P^{(N)})$ is a complete probability space. Let 
		$X^{(N)}$ and $X$ be random variables of $\Omega^{(N)} \mapsto P$ and $\Omega \mapsto P$, respectively. 
		Then, $X^{(N)}$ is convergent in distribution to $X$, denoted by $X^{(N)} \Rightarrow X$, if
		there exists $Z^{N}: \Omega \mapsto P$ 
		satisfying $\mathcal L(Z^{N}) = \mathcal L(X^{(N)})$, such that $Z^N \to X$ holds almost surely, i.e.
		$$\lim_{N\to \infty} d(Z^{N}, X) = 0,  \hbox{ almost surely in } \mathbb P,$$
		where $d$ represents the metric assigned to the space $P$.
	\end{proposition}

	In this paper, we  formulate the $N$-player game in the completed filtered probability space
	$$(\Omega^{(N)}, \mathcal F_T^{(N)}, \mathbb F^{(N)} := \{\mathcal F_t^{(N)}: 0\le t\le T\}, \mathbb P^{(N)}),$$
	and $Y^{(N)}$ 
	 is the continuous time Markov chain  
	in $\Omega^{(N)}$ with the same generator given by \eqref{eq:Y} and 
	$W^{(N)} = (W^{(N)}_{i}: i = 1, 2, \dots, N)$
	is an $N$-dimensional standard Brownian motion. We assume $Y^{(N)}$  and 	
	$W^{(N)}$ are independent of each other.

	For better clarity, we use the superscript $(N)$ for a random variable to emphasize the probability space 
	$\Omega^{(N)}$ it belongs to. 
	For example, Proposition \ref{p:conv} denotes a random variable in $\Omega^{(N)}$ by
	$X^{(N)}$, while its distribution copy in $\Omega$ by $Z^{N}$, but not by $Z^{(N)}$.


	\subsection{The equilibrium of MFGs}\label{s:mfg}
	In this section, we define the equilibrium of MFGs associated with a generic player's stochastic control problem in the probability setting $\Omega$, see Section \ref{s:notation}.
	
	Given a random measure flow $m: (0, T]\times \Omega \mapsto \mathcal P_2(\mathbb R)$, consider a generic player who wants to minimize her expected accumulated cost on $[0, T]$: 
	\begin{equation}
		\label{eq:total cost}
		\begin{array}{ll}
			J (y, x, \alpha)  = \displaystyle
			\mathbb E \left[  \int_0^T 
			\left( \frac{1}{2} \alpha_s^2 + F(Y_s, X_s, m_{s}) \right) 
			ds 
			+ G(Y_T, X_T, m_{T}) 
			\Big |  Y_0 = y, X_0 = x
			\right]
		\end{array}
	\end{equation} 
	with some given cost functions $F, G: \mathcal Y \times \mathbb R\times \mathcal P_2(\mathbb R) \mapsto \mathbb R$ and underlying random processes 
	$(Y, X): [0, T]\times \Omega \mapsto \mathcal Y \times \mathbb R$.
	Among three processes $(Y, X, m)$, the generic player can control the process $X$ via $\alpha$ in the form of
	\begin{equation}\label{eq:X}
		X_t = X_0 + \int_0^t \left( \tilde{b}_1(Y_s, s) X_s + \tilde{b}_2(Y_s, s)\alpha_s \right) ds + W_t, \quad \forall t \in [0, T],
	\end{equation}
	where $\tilde{b}_1(\cdot, \cdot)$ and $\tilde{b}_2(\cdot, \cdot)$ are two deterministic functions.
	We assume that  the initial state $X_0$ is independent of $Y$. 
	The Brownian motion $W$ is the individual noise of the generic player,
	the process $Y$ of \eqref{eq:Y} represents the common noise, and $m=(m_t)_{0\le t \le T}$ is a given random density flow normalized up to total mass one.
	
	The objective of the control problem for the generic player is to find its optimal control $\hat \alpha \in \mathcal A := L^4_{\mathbb F}$  to minimize the total cost, i.e.
	\begin{equation}
		\label{V_m}
		V[m](y, x) = J[m](y, x, \hat \alpha) \le J[m](y, x, \alpha), \quad \forall \alpha \in \mathcal A.
	\end{equation}
	Associated with the optimal control $\hat \alpha$, we denote the optimal path by 
	$\hat{X}=(\hat X_t)_{0\le t\le T}$. To introduce MFG Nash equilibrium, 
	it is often convenient to highlight the dependence of  
	the optimal path and optimal control of the generic player and its associated value on the underlying density  flow $m$, which are denoted by 
	$$\hat X_t[m], \hat \alpha_t [m],  \hbox{ and } V[m],$$ respectively.
	Now, we present the definition of the equilibrium below, see also  Volume II-P127 of \cite{carmona2018} 
	for a general setup with a common noise.
	\begin{definition}
		\label{d:ne}
		Given an initial distribution  $\mathcal L(X_0) = m_0 \in \mathcal P_2(\mathbb R)$, 
		a random measure flow $\hat m = \hat m(m_0)$ 
		is said to be an MFG equilibrium measure if it satisfies the fixed point condition
		\begin{equation}
			\label{eq:mhat}
			\hat m_t = \mathcal L(\hat X_t[\hat m]| Y), \ \forall 0 < t \le T, \ \hbox{ almost surely in } \mathbb P.
		\end{equation}
		The path $\hat X$ and the control $\hat \alpha$  associated to $\hat m$ 
		is called the MFG equilibrium path and equilibrium control, respectively.
		The value function of the control problem associated with the equilibrium measure $\hat m$ is 
		called as MFG value function, denoted by
		\begin{equation}
			\label{eq:U}
			U(m_0, y, x) = V[\hat m] (y, x).
		\end{equation}
	\end{definition}
	
	\begin{figure}[h]
	\centering
		\includegraphics[width= .5\textwidth]{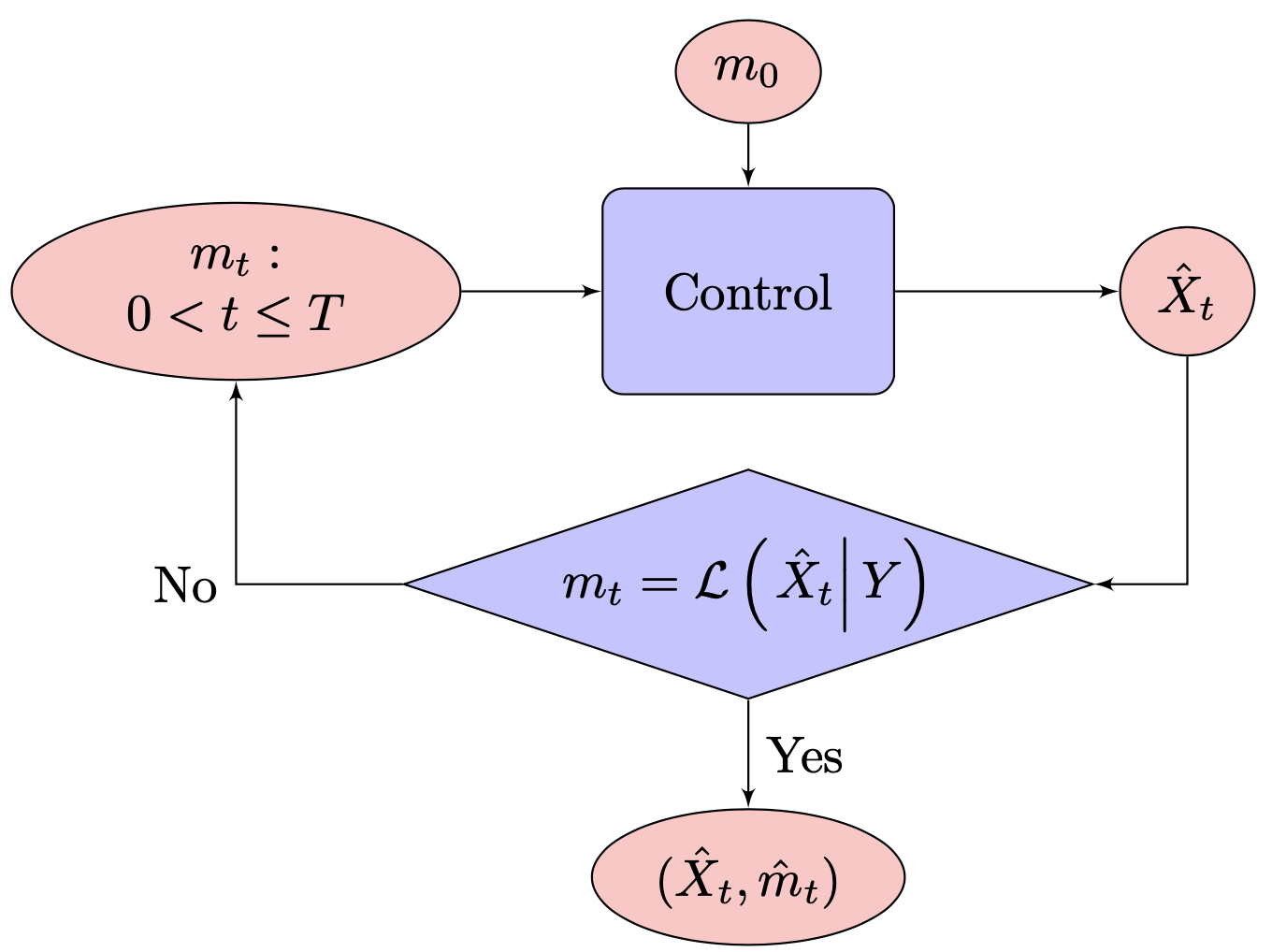}
		\caption{MFGs diagram.}
		\label{fig:MFG1}
	\end{figure}

	The flowchart of MFGs diagram is given in Figure \ref{fig:MFG1}. 
	It is noted from the optimality condition \eqref{V_m} and the fixed point condition \eqref{eq:mhat}  that 
	$$J[\hat m](y, x, \hat \alpha) \le J[\hat m](y, x, \alpha), \quad \forall \alpha$$
	holds for the equilibrium measure $\hat m$ and its associated equilibrium control $\hat \alpha$, 
	while it is not
	$$J[\hat m](y, x, \hat \alpha) \le J[m](y, x, \alpha), \quad \forall \alpha, m.$$
	Otherwise, this problem turns into a McKean-Vlasov control problem discussed in \cite{NNY20}.
	Furthermore, it's important to note that the Continuous-Time Markov Chain  $Y$ serves a role as common noise. 
	This is due to the fact that the mean field term is conditioned on the distribution of $Y$.


	\subsection{Equilibrium of the $N$-player game} \label{s:n-player}
	The discrete counterpart of MFGs is an $N$-player game, which is formulated below in the probability space $\Omega^{(N)}$, see Section \ref{s:notation} for more details on the probability setup. 
	
	Recall that, $W_{it}^{(N)}$ and $W_{jt}^{(N)}$ are independent Brownian motions for $j \neq i$ 
	and they are called individual noises in the $N$-player game.
	The common noise $Y^{(N)}$ is the continuous time Markov chain in $\Omega^{(N)}$ with the generator given by \eqref{eq:Y}. 
	Let the player $i$ follow the dynamic, 
	for $i = 1, 2, \dots, N$,
	\begin{equation}
		\label{eq:Xi}
		dX_{it}^{(N)} = \left(\tilde{b}_1(Y_t^{(N)}, t) X_{it}^{(N)} + \tilde{b}_2(Y_t^{(N)}, t) \alpha_{it}^{(N)} \right) dt + dW_{it}^{(N)}, \quad X_{i0}^{(N)} = x^{(N)}_i.
	\end{equation}

	The cost function for player $i$
	associated to the control $\alpha^{(N)} = (\alpha_i^{(N)}: i = 1, 2, \dots, N)$ is 
	\begin{equation}\label{eq:cost_N}
		\begin{aligned}
			J_i^{N}(y, x^{(N)}, \alpha^{(N)}) 
			&= \mathbb E\left[\int_0^T 
			\left(\frac{1}{2}|\alpha_{it}^{(N)}|^2 
			+ F(Y_t^{(N)}, X_{it}^{(N)}, \rho(X_t^{(N)})) \right) dt +  
			\right. \\ 
			& \hspace{1in} 
			\left. 
			G(Y_T^{(N)}, X_{iT}^{(N)}, \rho(X_T^{(N)}))
			\Big| X_{0}^{(N)} = x^{(N)}, Y_0^{(N)} = y  \right],
		\end{aligned}
	\end{equation}
	where $x^{(N)} = (x_1^{(N)}, x_2^{(N)}, \dots, x_N^{(N)})$ is an $\mathbb R^N$-valued random vector in $\Omega^{(N)}$ to denote
	 the initial state for $N$ player, $\alpha_i^{(N)} \in \mathcal A^{(N)} := L_{\mathbb F^{(N)}}^4$, and 
	$$\rho(x^{(N)}) = \frac{1}{N} \sum_{i=1}^N \delta_{x_i^{(N)}}$$ 
	is the empirical measure of a vector $x^{(N)}$ with Dirac measure $\delta$.
	We use the notation for the control
	$\alpha^{(N)} = (\alpha_{i}^{(N)}, \alpha_{-i}^{(N)}) = (\alpha_1^{(N)}, \alpha_2^{(N)}, \ldots, \alpha_N^{(N)})$.
	
	\begin{definition}\label{d:neN}
		\begin{enumerate}
			\item The value function of player $i$  for $i = 1, 2, \ldots, N$ of the Nash game is defined by
			$V^N = (V^N_i: i = 1, 2, \ldots, N)$ satisfying the equilibrium condition
			\begin{equation}
				\label{eq:value_i}
				\begin{array}
					{ll}
					V_i^{N}(y, x^{(N)}) & = J_i^{N} (y, x^{(N)}, \hat \alpha_i^{(N)}, \hat \alpha_{-i}^{(N)}) 
				   \le J_i^{N}(y, x^{(N)}, \alpha_i^{(N)}, \ \hat \alpha_{-i}^{(N)}), \quad \forall 
					\alpha_i^{(N)} \in \mathcal A^{(N)}.
				\end{array}
			\end{equation}
			
			\item
			The equilibrium path of the $N$-player game is the random path 
			$\hat X_t^{(N)} = (\hat X_{1t}^{(N)}, X_{2t}^{(N)}, \dots, \hat X_{Nt}^{(N)} )$ driven by 
			\eqref{eq:Xi} associated to the control $\hat \alpha_t^{(N)}$ 
			satisfying the equilibrium condition of \eqref{eq:value_i}.

		\end{enumerate}
	\end{definition}

	
		\subsection{The main result with quadratic cost structures}
	\label{s:main}
	We consider the following two functions
	$F, G: \mathcal Y \times \mathbb R\times \mathcal P_2(\mathbb R) \mapsto \mathbb R$ in the cost functional \eqref{eq:total cost}: 
	\begin{equation}
		\label{eq:running cost}
		F(y, x, m) = h(y) \int_{\mathbb R} (x-z)^2 m(dz),
	\end{equation}
	and
	\begin{equation}
		\label{eq:term cost}
		G(y, x, m) = g(y) \int_{\mathbb R} (x-z)^2 m(dz),
	\end{equation}
	for some $h, g: \mathcal Y \mapsto \mathbb R^+$. In this case, the  $F$ and $G$  terms
	in \eqref{eq:cost_N}  of the
	$N$-player game can be written by
	$$
	F(Y_t^{(N)}, X_{it}^{(N)}, \rho(X_t^{(N)}))  = 
	\frac{h(Y_t^{(N)})}{N}\sum_{j = 1}^N(X_{it}^{(N)}-X_{jt}^{(N)})^2, $$
	and
	$$
	G(Y_T^{(N)}, X_{iT}^{(N)}, \rho(X_T^{(N)})) = 
	\frac{g(Y_T^{(N)})}{N}\sum_{j = 1}^N(X_{iT}^{(N)}-X_{jT}^{(N)})^2,
	$$
	respectively. 

 \begin{remark}
    First, we note that $F$ and $G$ possess the quadratic structures in $x$. Secondly, 
	the coefficients $h(y)$ and $g(y)$ provide the sensitivity to the mean field effects, which depend on the current CTMC state. For another remark, let us consider the scenario where the number of states is $2$ and sensitivities are invariant, say
	$$h(0) = h(1) = h, \ g(0) = g(1) = 0.$$ 
	Then the cost function and hence the entire problem is free from the common noise. 
	Interestingly, as shown in the Appendix \ref{s:no_common}, 
	there is no global solution for MFGs when $h<0$, while there is a global solution when $h>0$.

    Moreover, the uniqueness of Mean Field Game can be achieved under the displacement monotonicity condition. It is easy to check that \eqref{eq:running cost}-\eqref{eq:term cost} 
    satisfy
    the displacement monotonicity condition. Note that
    $$F_x(y, x, m) = 2h(y) (x - [m]_1), \quad G_x(y, x, m) = 2g(y) (x - [m]_1),$$
    which gives that
    \begin{equation*}
        \mathbb E \left[\left(F_x (y, X_1, m_{X_1}) - F_x (y, X_2, m_{X_2}) \right) (X_1 - X_2) \right] 
        = 2h(y) \left(\mathbb E \left[ \left(X_1 - X_2 \right)^2 \right] - \left(\mathbb E[X_1] - \mathbb E[X_2] \right)^2 \right) \geq 0
    \end{equation*}
    for all $y \in \mathcal Y$ if $h > 0$ on $\mathcal Y$, where $m_{X_1}$ and $m_{X_2}$ is the law of $X_1$ and $X_2$ respectively. Similarly, we can obtain that 
    $$\mathbb E \left[\left(G_x (y, X_1, m_{X_1}) - G_x (y, X_2, m_{X_2}) \right) (X_1 - X_2) \right] \geq 0$$
    for all $y \in \mathcal Y$ if $g > 0$ on $\mathcal Y$. Therefore, we require positive values for all sensitivities for simplicity. It is of course an interesting problem to investigate the explosion 
	when some sensitivities are negative.

 \end{remark}

	Wrapping up the above discussions, we impose the following assumptions: 
	\begin{itemize}
	    \item [(A0)]
		$\tilde{b}_1(y, \cdot), \tilde{b}_2(y, \cdot): [0, T]  \mapsto  \mathbb R$ 
		are continuous functions for all $y\in \mathcal Y$.
		\item [(A1)]
		The cost functions are given by \eqref{eq:running cost}-\eqref{eq:term cost} with 
		$h, g>0$; The initial $X_0$ of MFGs satisfies 
		$\mathbb E[X_0^2]<\infty$.
		\item
		[(A2)] In addition to (A1), 
		the initial $x^{(N)} = (x_1^{(N)}, x_2^{(N)}, \dots, x_N^{(N)})$
		of the $N$-player game is a vector of i.i.d. random variables in $\Omega^{(N)}$ with the same distribution 
		as the initial  $\mathcal L(X_0)$ of MFG.
	\end{itemize}
	Our objective for this paper is to understand
	the Nash equilibrium of MFGs and its connection to 
	the $N$-player game equilibrium:
	\begin{itemize}
		\item [(P1)]
		With 
		Assumptions
		(A0), (A1), and (A2), obtain the convergence rate of 
		$(\hat X_{1t}^{(N)}, Y^{(N)})$ from the $N$-player game of Definition \ref{d:neN} 
		to $(\hat X_t, Y)$ from MFGs of Definition \ref{d:ne} in distribution.
	\end{itemize}
To answer (P1), it is critical to have a solid understanding of the joint distribution $(\hat X_t, Y)$ for the underlying MFG, which yields 
another question:
\begin{itemize}
		\item [(P2)]
		With 
		Assumptions
		(A0) and (A1), 
		characterize the MFG equilibrium path $\hat X$, 
		as well as associated equilibrium measure $\hat m$ along the Definition \ref{d:ne};
\end{itemize}
	
For our first main result, we first answer (P2) via the following Riccati system for unknowns 
$(a_y, b_y, c_y, k_y : y \in \mathcal Y)$:
	\begin{equation}
		\label{eq:ode1}
		\begin{aligned}
			\begin{cases}
				\displaystyle a_y' + 2\tilde{b}_{1y} a_y - 2\tilde{b}^2_{2y} a_y^2 + \sum_{i=1}^{\kappa} q_{y,i} a_i + h_y = 0, \\
				\displaystyle b_y' + \left(2 \tilde{b}_{1y} - 4\tilde{b}^2_{2y} a_y \right) b_y + \sum_{i=1}^{\kappa} q_{y,i} b_i + h_y  =0, \\
				\displaystyle c_y' + a_y + b_y + \sum_{i=1}^{\kappa} q_{y,i} c_i  = 0,\\
				\displaystyle k_y' -2\tilde{b}^2_{2y} a_y^2 + 4 \tilde{b}^2_{2y} a_y b_y + 2 \tilde{b}_{1y} k_y + \sum_{i=1}^{\kappa} q_{y,i} k_i = 0, \\
				\displaystyle a_y(T) = b_y(T) = g_y\text{ , } c_y(T) =  k_y(T) = 0, 
			\end{cases}
		\end{aligned}
	\end{equation}
	where $h_y = h(y), \, g_y = g(y)$ for $y \in \mathcal Y$. 
	Next, we present our first main result about the equilibrium path, the equilibrium control, and the value function in MFG.

	\begin{theorem}[MFG]
		\label{t:main}
		Under  (A0)-(A1), there exists a unique solution 
		$(a_y, b_y, c_y, k_y: y \in \mathcal Y)$ for the Riccati system \eqref{eq:ode1}. 
		With these solutions,
		the MFG equilibrium path $\hat X = \hat X [\hat m]$ is given by
		\begin{equation}
			\label{eq:Xhat03}
			d \hat X_t = \left( \tilde{b}_1(Y_t, t) \hat{X}_t - 2 \tilde{b}_2^2(Y_t, t)  a_{Y_t}(t) \left(\hat X_t - \hat{\mu}_t \right) \right)  dt + dW_t, \quad \hat X_0 = X_0,
		\end{equation}
		with equilibrium control
		\begin{equation}
			\label{eq:alpha011}
			\hat \alpha_t = - 2 \tilde{b}_2(Y_t, t) a_{Y_t}(t) \left(\hat X_t - \hat{\mu}_t \right),
		\end{equation}
		where 
		$$d \hat{\mu}_t = \tilde{b}_1(Y_t, t) \hat{\mu}_t dt, \quad \hat{\mu}_0 = \mathbb E[X_0].$$
		Moreover, the value function $U$ 
		is 
		$$U (m_0, y, x) =a_y(0) x^2 {-2a_y(0)x[m_0]_1 + k_y(0)[m_0]_1^2 }+ b_y(0) [m_0]_2 + c_y(0), \quad  y \in \mathcal Y.$$
	\end{theorem}
 
The proof of theorem \ref{t:main} is based on the Markovian structure of the equilibrium and the fixed point condition of the MFG problem, and it is provided in Subsection \ref{s:proof_of_main_result}. The next theorem establishes the convergence result and answers the problem (P1) with the convergence rate $\frac{1}{2}$.

	\begin{theorem}[Convergence rate]
		\label{t:main02} 
		Under 
		Assumptions
		(A0)-(A1)-(A2), the joint law
		$(\hat X_{1t}^{(N)}, Y_t^{(N)})$ of the $N$-player game converges in distribution to that of the MFG equilibrium $(\hat X_t, Y_t)$ 
		for any $t\in (0, T]$ at the convergence rate 
		$$\mathbb W_2 \left(\mathcal L(\hat X_{1t}^{(N)}, Y_t^{(N)}), \mathcal L (\hat X_t, Y_t) \right) = 
	    O \left(N^{- \frac 1 2} \right), \quad \hbox{ as } N\to \infty.$$
	\end{theorem}

 The proof of Theorem \ref{t:main02} is given in Subsection \ref{s:convergence} since it needs the comparison between the equilibrium path $\hat X_{1t}^{(N)}$ in $N$-player game and the equilibrium path $\hat X_t$ in MFG.

	\subsection{Remarks on the main results}
	\label{s:remarks on main}
	One can interpret the main results in plain words:
	For $N$-player game with dynamic \eqref{eq:Xi} and cost structure \eqref{eq:cost_N} for large $N$,
	the equilibrium control of the generic player can be effectively approximated by steering itself toward the population center 
	$\hat \mu_t$ depending only on the function $\tilde b_1(\cdot)$ and  the entire past of the common noise, 
	whose velocity is dependent on only the function $\tilde b_2(\cdot)$ and the entire past of the common noise.
	The effectiveness can be quantified by the convergence rate of $1/2$ for the one-dimensional 
	Mean Field Game
	under LQ structure and CMTC common noise. 
	A natural question is whether the convergence rate can be generalized to more general settings.
	
	This paper focuses on the one-dimensional problem to avoid unnecessary symbol complexity. 
	Therefore, the main convergence rate $1/2$ still holds for multidimensional problems using the same coupling procedure. 
	For convenience to check, we summarize the computation involved in multidimensional problems
	in Appendix \ref{s:section5}.
	
	The current coupling procedure can also be adapted with suitable modifications 
	to the LQ Mean Field Game problems with Brownian common noise, 
	see \cite{JSY23}.
	In particular, the reduction of the $O(N^3)$-dimensional ODE can be conducted similarly
	and the convergence rate is still maintained as $1/2$.
	However, the dependence of the mean and variance process on the common noise and subsequent calculations 
	are significantly different from the current paper, 
	see Definition 4 of \cite{JSY23}.
	
	Indeed, choosing the CTMC common noise instead of Brownian motion does not
	simplify the underlying problem, since it preserves the path-dependence feature of the equilibrium measure.
	On the contrary, the advantage of CTMC common noise is that the applications aim to model less 
	frequently changing environment settings, such as government policies implemented by multiple different regimes.
	Due to its realistic applications, stochastic control theory perturbed by CTMC is 
	extensively studied in the context of hybrid control problems, see books \cite{MY06, YZ10} and the references therein.

	We close this section with a remark on the uniqueness. 
	The uniqueness of Mean Field Game can be achieved under Lasry-Lions monotonicity \cite{LJ07} 
	or displacement monotonicity \cite{GMMZ21} and our setting in Section \ref{s:mfg} satisfies the displacement monotonicity. Thus, the convergence of Theorem \ref{t:main02} 
	implies that the unique equilibrium path of $N$-player game converges 
	to the unique equilibrium paths of the limiting MFG, which is characterized by Theorem \ref{t:main}.


	\section{
	Main results of MFG
	}
	\label{s:section3}
	
	This section is devoted to the proof of the first main result Theorem \ref{t:main} on the MFG solution. 
	First, we outline the scheme based on the Markovian structure of the equilibrium by reformulating the MFG problem in Subsection~\ref{s:overview}.
	Next, we solve the underlying control problem in Subsection~\ref{s:Riccati system} and provide the corresponding Riccati system.
	Finally, Subsection~\ref{s:proof_of_main_result} proves Theorem \ref{t:main} by checking the fixed point condition of MFG problem.
	\subsection{Overview}\label{s:overview}
	By Definition \ref{d:neN}, to solve for the equilibrium measure, one shall search the infinite dimensional space of the random measure flows  $m:(0, T]\times \Omega \mapsto \mathcal P_2(\mathbb R)$, until a measure flow satisfies the fixed point condition $m_t = \mathcal L (\hat X_t|Y), \forall t\in(0,T]$, 
	see Figure \ref{fig:MFG1}, which requires to check the following infinitely many conditions:
	$$ [m_t]_k = \mathbb E[\hat X^k_t |Y], \quad \forall k = 1, 2, \ldots,$$
	if they exist. 
	
	The first observation is that the cost functions $F$ and $G$  in  \eqref{eq:running cost}-\eqref{eq:term cost} are dependent on the measure $m$ only via the first two moments: 
	\begin{equation*}
		\begin{aligned}
		    F(y, x, m) &= h(y) (x^2 - 2x [m]_1 + [m]_2), \\
		    G(y, x, m) &= g(y) (x^2 - 2x [m]_1 + [m]_2).
		\end{aligned}
	\end{equation*}
	Therefore, the underlying stochastic control problem for MFGs can be entirely determined by the input given by 
	$\mathbb R^2$ valued random process  $\mu_t = [m_t]_1$ and $\nu_t = [m_t]_2$, 
	which implies that the fixed point condition can be effectively reduced to check two conditions only:
	$$\mu_t = \mathbb E[\hat X_t|Y], \ \nu_t = \mathbb E[\hat X_t^2|Y].$$
	This observation effectively reduces our search from the space of random measure-valued processes 
	$m: (0, T]\times \Omega \mapsto \mathcal P_2(\mathbb R)$ to the space of $\mathbb R^2$-valued random processes
	$(\mu, \nu): (0, T]\times \Omega \mapsto \mathbb R^2$. 
	
	Note that, if underlying MFGs have no common noise $Y$, then $(\mu, \nu)$ is a deterministic 
	mapping $[0, T]\mapsto \mathbb R^2$ and the above observation 
	is enough to reduce the original infinite-dimensional MFGs into a finite-dimensional system.
	However,  the following example shows that this is not the case 
	for MFGs with a common noise and it becomes the main drawback to characterizing MFGs via a finite-dimensional system.
	\begin{example}\label{exm:1}
	To illustrate, we consider the following uncontrolled mean field dynamics:
	Let the mean field term $\mu_t := \mathbb E[\hat X_t|Y]$, where the underlying dynamic is given by
	$$d \hat X_{t} = - \mu_t Y_t dt + dW_{t}.$$
	\begin{itemize}
		\item $\mu_t$ is path dependent on $Y$, i.e.
		$$\mu_t = \mu_0  \exp \Big\{ - \int_0^t Y_s ds \Big\}.
		$$
		This implies that no finite dimensional system is possible to  characterize the process $\mu_t$, since 
		the $(t, Y) \mapsto \mu_t$ is a function on an infinite dimensional domain.
		\item  $\mu_t $ is {\it Markovian}, i.e.
		$$ d \mu_t = -  Y_t \mu_t dt.$$
		It might be possible to 
		characterize $\mu_t$ via a function $(t, Y_t, \mu_t) \mapsto \frac{d \mu_t}{dt}$ on a finite dimensional domain.
		
	\end{itemize}
	\end{example}

		To solidify the above idea, we need to postulate the Markovian structure for the first and second moments of the MFG equilibrium.
	More precisely, our search  for the equilibrium will be confined to the space $\mathcal M$ of measure flows whose first and second moment exhibits Markovian structure. 
	\begin{definition}
		\label{d:markov}
		The space $\mathcal M$ is the collection of all $\mathcal F^Y_t$-adapted measure flows 
		$m: [0,T]\times \Omega \mapsto \mathcal P_2(\mathbb R)$, whose first moment  $[m_t]_1 := \mu_t$ and second moment $[ m_t]_2 := \nu_t$ satisfy 
		\begin{equation}
			\label{eq:mu_sigma}
			\begin{aligned}
				&\mu_t = \mu_0 + \int_0^t \left(w_0(Y_s,s)\mu_s +w_1(Y_s,s)\right)ds, \\
				& \nu_t = \nu_0 + \int_0^t \left(w_2(Y_s,s)\mu_s+w_3(Y_s,s)\nu_s +w_4(Y_s,s)\mu_s^2 + w_5(Y_s,s)\right)ds,
			\end{aligned}
		\end{equation}
		for all $t \in [0, T]$ and for some smooth deterministic functions $(w_i: i = 0, 1, \ldots, 5)$.
	\end{definition}
	
	\begin{figure}[h]
	\centering
		\includegraphics[width= .5\textwidth]{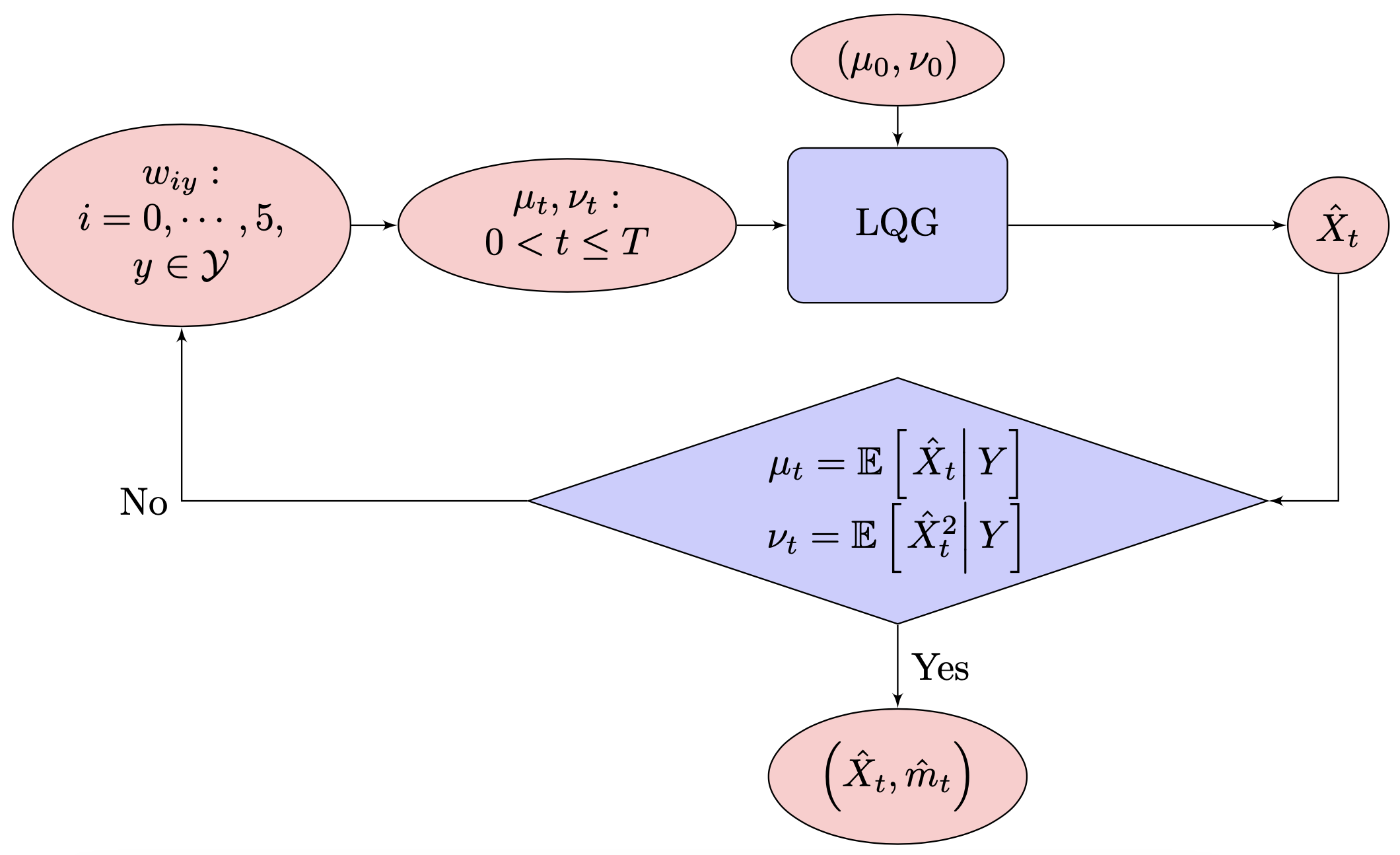}
		\caption{ Equivalent MFGs diagram with $\mu_0 = [m_0]_1$ and $\nu_0 = [m_0]_2$.}
		\label{fig:MFG2}
	\end{figure}

	The flowchart for our equilibrium is depicted in Figure \ref{fig:MFG2}. 
	Subsection \ref{s:Riccati system} 
	covers the derivation of the Riccati system for the LQG system with a given population measure flow $m\in \mathcal M$, 
	which provides the key building block to MFGs. 
	In Subsection \ref{s:proof_of_main_result}, we check the fixed point condition and provide a finite-dimensional characterization of MFGs, which gives the first main result Theorem \ref{t:main}.

	\subsection{The generic player's control with a given population measure} 
	\label{s:Riccati system}
	
	The advantage of the generic player's control problem associated with $m\in \mathcal M$ is that 
	its optimal path can be characterized via the following classical stochastic control problem:
	\begin{itemize}
		\item (P3) 
		Given smooth functions $w = (w_i: i = 0, 1, \ldots, 5)$, find the optimal value 
		$\bar V = \bar V[w]$
		$$
		\begin{array}{ll}
			\bar V(y, x, t, \bar \mu, \bar v) & \displaystyle = \inf_{\alpha\in \mathcal A}
			\mathbb E \left[  \int_t^T \left(\frac 1 2 \alpha_s^2 + \bar F(Y_s, X_s, \mu_{s}, \nu_s) \right)ds 
			\right.
			\\ & \displaystyle \hspace{0.5in} 
			\left. 
			+ \bar G(Y_T, X_T, \mu_T, \nu_T) 
			\right\vert Y_t = y, X_t = x, \mu_t = \bar \mu, \nu_t = \bar \nu
			\Big]
		\end{array}
		$$
		underlying $\mathbb R^4$-valued processes $(Y, X, \mu, \nu)$ defined through \eqref{eq:Y}-\eqref{eq:X}-\eqref{eq:mu_sigma} with the finite dimensional cost functions: $\bar F, \bar G: \mathbb R^4 \mapsto \mathbb R$ given by
		$$
		\bar F (y, x, \bar \mu, \bar \nu) = h(y) (x^2 - 2x \bar \mu + \bar \nu),$$
		$$ 
		\bar G (y, x, \bar \mu, \bar \nu) = g(y) (x^2 - 2x \bar \mu + \bar \nu),
		$$
		where $\bar{\mu}, \bar{\nu}$ are scalars, while $\mu, \nu$ are used as processes.
	\end{itemize}
	
	\begin{lemma}\label{l:mkv}
		Given $m\in \mathcal M$ associated with $w = (w_i: i = 0, 1, \dots, 5)$, 
		the player's value \eqref{V_m} under assumption (A1) is 
		$$U[m_0] (y, x) = \bar V(y, x, 0, [m_0]_1, [m_0]_2),$$
		and the optimal control has a feedback form
		$$\hat \alpha_t = \bar \alpha(Y_t, X_t, t, \mu_t, \nu_t)$$
		underlying the processes $(Y, X, \mu, \nu)$ defined through \eqref{eq:Y}-\eqref{eq:X}-\eqref{eq:mu_sigma}, 
		whenever there exists a feedback optimal control $\bar \alpha$ for the problem (P3).
	\end{lemma}
	\begin{proof}
		Due to the quadratic cost structure in  \eqref{eq:running cost}-\eqref{eq:term cost}, we have enough regularity to all concerned value functions and the details are omitted. 
	\end{proof}
	Next, we turn to the solution to the control problem (P3).
	
	\subsubsection{HJB equation}
	
	For the simplicity of notations, for each $i \in \{0, 1, 2, 3, 4, 5\}$ and $y \in \mathcal Y$, denote 
	the function $(x,t, \bar \mu, \bar \nu) \mapsto v(y,x,t, \bar \mu, \bar \nu)$ as $v_y$, and denote $t\mapsto w_i(y, t)$ as $w_{iy}$. We apply similar notations for other functions whenever they have a variable $y\in \mathcal Y$.
	Formally, under enough regularity conditions, the value function $\bar V$ defined in (P3) is the solution $v$ of the following coupled HJBs
	\begin{equation}
		\label{eq:hjb1}
		\begin{aligned}
			\begin{cases}
				\displaystyle \partial_t v_y + \tilde{b}_{1y}x \partial_x v_y - \frac{1}{2} \left(\tilde{b}_{2y} \partial_x v_y \right)^2 + \frac{1}{2}\partial_{xx} v_y + \partial_\mu v_y \left(w_{0y}\bar \mu + w_{1y} \right)+ \\ 
				\hspace{1in} \displaystyle  \partial_\nu v_y \left(w_{2y}\bar \mu+w_{3y}\bar \nu +w_{4y}\bar \mu^2 + w_{5y}\right) 
				+ \sum_{i=1}^{\kappa} q_{y,i} v_i + \bar F_y = 0, \\
				\displaystyle v_{y}(x, T, \mu_T,\nu_T) = \bar G_{y}(x, \mu_T,\nu_T), \ y \in \mathcal Y.
			\end{cases}
		\end{aligned}
	\end{equation}
	Furthermore, the optimal control has to admit the feedback form of
	\begin{equation}
		\label{eq:optimal control}
		\hat \alpha(t) = - \tilde{b}_2(Y_t, t) \partial_x v({Y_t}, \hat{X_t}, t, \mu_t,  \nu_t).
	\end{equation}		
	Next, we identify what conditions are needed for equating the control problem (P3) and HJB equation. Denote 
	$$\mathbb S = \left\{ v \in C^{\infty} : \begin{array}
	{ll}
	(1+|x|^2)^{-1} (|v| + |\partial_t v|) + 
	\\ \hspace{.5in} (1+|x|)^{-1}(|\partial_x v| + |\partial_\mu v| + |\partial_\nu v|) +  |\partial_{xx} v| 
	< K, \forall (y, x, t, \mu, \nu), \hbox{ for some } K
	\end{array}
	    \right\}.$$
	\begin{lemma}
		\label{l:verification}  (Verification theorem)
		Consider  the control problem (P3) with some given smooth $w$.
			Suppose there exists a solution $v \in \mathbb S$ of  \eqref{eq:hjb1}. Then,  $v_y(x, t,  \bar \mu,\bar \nu) = \bar V(y, x, t,\bar \mu,\bar \nu)$ holds, and an optimal control is provided by \eqref{eq:optimal control}.
	\end{lemma}
	\begin{proof}
			We first prove the verification theorem. Since $v \in \mathbb S$, for any admissible $\alpha\in L_{\mathbb F}^4$, the process $X^\alpha$ is well defined and one can use Dynkin's formula given by Lemma \ref{l:chain rule} to write
			$$\mathbb E \left[v(Y_T,X_T, T,\mu_T,\nu_T)\right] = v(y, x, t, \bar \mu,\bar \nu) 
			+ \mathbb E \left[ \int_t^T \mathcal G^{\alpha(s)} v(Y_s, X_s, s, \mu_s, \nu_s) ds \right],$$
			where
			\begin{equation*}
				\begin{aligned}
					\mathcal G^a f(y, x, s, \bar{\mu}, \bar{\nu}) &= \left(\partial_t + \left(\tilde{b}_{1y} x + \tilde{b}_{2y}a \right)\partial_x + \frac 1 2 \partial_{xx} + \mathcal Q + \left(w_{0y}\bar \mu+w_{1y}\right) \partial_{\bar\mu} + \right. \\
					& \quad  \quad  \quad \left. \left(w_{2y}\bar \mu + w_{3y}\bar \nu+w_{4y}\bar \mu^2 + w_{5y}\right) \partial_{\bar\nu} \right) f(y, x, s, \bar\mu, \bar\nu).
				\end{aligned}
			\end{equation*}
			Note that HJB actually implies that
			$$\inf_a \left\{ \mathcal G^a v + \frac 1 2 a^2\right\} = - \bar F,$$
			which again implies
			$$- \mathcal G^a v \le \frac 1 2 a^2 + \bar F.$$
			Hence, we obtain that for all $\alpha\in L_{\mathbb F}^4$,
			$$
			\begin{array}
				{ll}
				\quad v(y, x, t, \bar \mu,\bar \nu) 
				\\ \displaystyle =  \mathbb E \left[ \int_t^T  - \mathcal G^{\alpha(s)} v(Y_s, X_s, s, \mu_s,\nu_s) ds \right]  +\mathbb E \left[v(Y_T,X_T, T,\mu_T,\nu_T) \right] 
				\\ \displaystyle 
				\le  \mathbb E \left[ \int_t^T \left( \frac 1 2 \alpha^2(s) + \bar F(Y_s, X_s,\mu_s,\nu_s) \right) ds \right] +\mathbb E \left[\bar G(Y_T,X_T, \mu_T, \nu_T) \right] \\
				= J(y, x, t,  \alpha, \bar \mu,\bar \nu).
			\end{array}
			$$
			
			In the above, if $\alpha$ is replaced by $\hat \alpha$ given by the feedback form \eqref{eq:optimal control},
			then since $\partial_x v$ is Lipschitz continuous in $x$, there exists corresponding optimal path $\hat X \in L_{\mathbb F}^4$.
			Thus, $\hat \alpha$ is also in $L_{\mathbb F}^4$. One can repeat all above steps by replacing $X$ and $\alpha$ by $\hat X$ and $\hat \alpha$, and $\le $ sign by
			$=$ sign to conclude that $v$ is indeed the optimal value.

	\end{proof}
	
	\subsubsection{LQG solution}
	Note that, the costs $\bar F$ and $\bar G$ of (P3) are quadratic functions in $(x, \bar \mu, \bar \nu)$, while the
	drift function of the process $\nu$ of \eqref{eq:mu_sigma} is not linear in $(x, \bar \mu, \bar \nu)$. Therefore, the  control problem (P3) does not fall into the standard LQG control framework. Nevertheless, similar to the LQG solution, we guess the value function as a quadratic function in the form of
	\begin{equation}
		\label{eq:v1}
		\begin{aligned}
			v_y(x, t,\bar \mu,\bar \nu) =& a_y(t) x^2 +d_y(t) x+e_y(t)\bar \mu +f_y(t)x\bar \mu +   k_y(t)\bar \mu^2 +  b_y(t) \bar \nu + c_y(t), \quad y \in \mathcal{Y}.
		\end{aligned}
	\end{equation}
	With the above setup, for $t \in [0, T]$, the optimal control is
	\begin{equation}
		\label{eq:alpha}
		\hat \alpha_t = - \tilde{b}_2(Y_t, t) \partial_x v(Y_t, \hat X_t, t, \mu_t, \nu_t) = - \tilde{b}_2(Y_t, t) \left( 2 a_{Y_t}(t) \hat X_t + d_{Y_t}(t) + f_{Y_t}(t)\mu_t \right),
	\end{equation}
	and the optimal path $\hat X$ is
	\begin{equation}
		\label{eq:Xhat}
		d \hat X_t = \left( \tilde{b}_1(Y_t, t) \hat{X}_t - \tilde{b}_2^2(Y_t, t) \left( 2 a_{Y_t}(t) \hat X_t + d_{Y_t}(t) + f_{Y_t}(t)\mu_t \right) \right) dt + dW_t.
	\end{equation}
	Denote the following ODE systems for $y \in \mathcal Y$,
	\begin{equation}
		\label{eq:ode2}
		\begin{aligned}
			\begin{cases}
				\displaystyle a_y' + 2\tilde{b}_{1y}a_y - 2 \tilde{b}_{2y}^2 a_y^2 + \sum_{i=1}^{\kappa} q_{y,i} a_i + h_y = 0, \\
				\displaystyle d_y' + \tilde{b}_{1y} d_y - 2 \tilde{b}_{2y}^2 a_y d_y + f_y w_{1y} + \sum_{i=1}^{\kappa} q_{y,i} d_i = 0 , \\
				\displaystyle e_y' -\tilde{b}_{2y}^2 d_y f_y +2k_y w_{1y}+ e_y w_{0y}+b_y w_{2y} + \sum_{i=1}^{\kappa} q_{y,i} e_i = 0,\\
				\displaystyle f_y' + \tilde{b}_{1y} f_y - 2 \tilde{b}_{2y}^2 a_y f_y + f_y w_{0y} + \sum_{i=1}^{\kappa} q_{y,i} f_i - 2h_y  = 0, \\
			    \displaystyle k_y' -\frac{1}{2} \tilde{b}_{2y}^2 f_y^2 + 2k_y w_{0y} + b_y w_{4y} + \sum_{i=1}^{\kappa} q_{y,i} k_i= 0, \\
				\displaystyle b_y' +b_y w_{3y} + \sum_{i=1}^{\kappa} q_{y,i} b_i + h_y = 0, \\
				\displaystyle c_y' + a_y -\frac{1}{2}\tilde{b}_{2y}^2 d_y^2 + e_y w_{1y} + b_y w_{5y} + \sum_{i=1}^{\kappa} q_{y,i} c_i = 0,
			\end{cases}
		\end{aligned}        
	\end{equation}
	with terminal conditions
	\begin{equation}
		\label{eq:ode2_terminal}
		\begin{aligned}
			&a_y(T) = g_y, \ b_y(T) = g_y, \  c_y(T) = 0, \ d_y(T) = 0, \ e_y(T) = 0, \ f_y(T) = -2g_y, \ k_y(T) = 0.
		\end{aligned}
	\end{equation}
	
	\begin{lemma}
		\label{l:ricatti1}
		Suppose there exists a unique solution $(a_y, b_y, c_y, d_y, e_y, f_y,k_y: y \in \mathcal Y)$ to the ODE system \eqref{eq:ode2}-\eqref{eq:ode2_terminal} on $[0,T]$. Then the value function of (P3) is
		\begin{equation}
			\label{eq:Vhat}
			\begin{aligned}
				&\bar V (y, x, t, \bar \mu, \bar \nu) = v_y(x, t, \bar \mu, \bar \nu) \\ 
				=\ & a_y(t) x^2 +d_y(t) x+e_y(t)\bar \mu +f_y(t)x\bar \mu + k_y(t)\bar \mu^2 +  b_y(t) \bar \nu + c_y(t)
			\end{aligned}
		\end{equation}
		for $y \in \mathcal Y$ and the optimal control and optimal path are given by \eqref{eq:alpha} and \eqref{eq:Xhat}, respectively.
	\end{lemma}
	\begin{proof}
		With the form of value function $v_{y}$ given in \eqref{eq:v1} and the first and second moment of the conditional population density given in \eqref{eq:mu_sigma}, we have
		\begin{equation*}
			\begin{aligned}
				&\partial_t v_{y}=a_y'(t) x^2 +d_y'(t) x+e_y'(t)\bar \mu+f_y'(t)x\bar \mu + k_y'(t)\bar \mu^2 +  b_y'(t) \bar \nu + c_y'(t),\\
				& \partial_x v_{y}=2xa_{y}(t)+d_{y}(t)+f_y(t)\bar \mu,\\
				& \partial_{xx} v_{y}=2a_{y}(t),\\
				& \partial_{\bar\mu} v_y = e_y(t)+f_y(t) x + 2k_y(t)\bar \mu, \\
				& \partial_{\bar\nu} v_y = b_y(t), 
			\end{aligned}
		\end{equation*}
		for $y \in \mathcal Y$. Plugging them back to the coupled HJBs in \eqref{eq:hjb1}, we get a system of ODEs in \eqref{eq:ode2} by equating $x$, $\bar \mu$, $\bar \nu$-like terms in each equation.
		
		Therefore, any solution $(a_y, b_y, c_y,d_y, e_y, f_y, k_y: y \in \mathcal Y)$ of ODE system \eqref{eq:ode2} leads to the solution of HJB \eqref{eq:hjb1} in the form of the quadratic function given by \eqref{eq:Vhat}.
		Since the $(a_y, b_y, c_y, d_y, e_y, f_y, k_y: y \in \mathcal Y)$ are differentiable functions on the closed set $[0, T]$, they are also bounded, and the function $v$ meets regularity conditions 
		required by Lemma \ref{l:verification} to conclude the desired result.
	\end{proof}

	\subsection{Fixed point condition and the proof of Theorem \ref{t:main}}
	\label{s:proof_of_main_result}
	Going back to the ODE  system  \eqref{eq:ode2}, there are $7 \kappa$ equations, while we have total $13 \kappa$ deterministic functions of $[0, T]\times \mathbb R$ to be determined to characterize MFGs. Those are
	$$(a_y, b_y, c_y, d_y, e_y, f_y, k_y: y \in \mathcal Y) \hbox{ and } (w_{iy}: i = 0, 1, \dots 5, \ y \in \mathcal Y).$$
	In the following,
	we identify the missing $6 \kappa$ equations by checking 
	the fixed point condition:
	\begin{equation}\label{eq:fpc}
		\mu_s = \mathbb E \left[\left.\hat{X}_s \right\vert Y\right], \ \nu_s = \mathbb E \left[\left.\hat{X}_s^2 \right\vert Y\right], \quad \forall s \in [0, T],
	\end{equation}
	where $\mu$ and $\nu$ are two auxiliary processes $(\mu,\nu)[w]$ defined in \eqref{eq:mu_sigma}, see Figure \ref{fig:MFG2}. 
	This leads to a complete characterization of
	the equilibrium for the MFG posed by (P2).
	
	Note that based on the dynamic of the optimal $\hat{X}$ defined in \eqref{eq:Xhat}, 
	the fixed point condition \eqref{eq:fpc} implies that
	the first moment $\hat \mu_s := \mathbb E \left[\left.\hat{X}_s \right\vert Y\right]$  
	and the second moment $\hat \nu_s := \mathbb E \left[\left.\hat{X}_s^2 \right\vert Y\right]$  
	of the optimal path conditioned on $Y$ satisfy
	\begin{equation}
		\label{eq:mu_sigma_gen}
		\begin{aligned}
			\begin{cases}
				\displaystyle \hat \mu_s = \bar\mu +
				\int_t^s \left( \left( \tilde{b}_1(Y_r, r) - \tilde{b}_2^2(Y_r,r) \left( 2a_{Y_r}(r) + f_{Y_r}(r)\right) \right) \hat \mu_r - \tilde{b}_2^2(Y_r, r) d_{Y_r}(r) \right) dr, \\
				\displaystyle 
				\hat \nu_s = \bar\nu + \int_t^s  \left(1+2\tilde{b}_1(Y_r, r) \hat \nu_r -  \tilde{b}_2^2(Y_r, r) \left( 4a_{Y_r}(r) \hat \nu_r + 2 d_{Y_r}(r) \hat\mu_r + 2 f_{Y_r}(r)\hat\mu_r^2 \right)\right)  dr,
			\end{cases}
		\end{aligned}
	\end{equation}
	for $s\ge t$. Note that under the optimal control in \eqref{eq:alpha}, comparing the terms in \eqref{eq:mu_sigma} and  \eqref{eq:mu_sigma_gen}, we obtain another $6 \kappa$ equations:
	\begin{equation}\label{eq:w12}
	\begin{aligned}
		& w_{0y} = \tilde{b}_{1y} - 2 \tilde{b}_{2y}^2 a_y - \tilde{b}_{2y}^2 f_y, \ w_{1y} = - \tilde{b}_{2y}^2 d_y, \ w_{2y} = -2 \tilde{b}_{2y}^2 d_y, \\
		& w_{3y} = -4 \tilde{b}_{2y}^2 a_y + 2 \tilde{b}_{1y}, \ w_{4y} = -2 \tilde{b}_{2y}^2 f_y, \ w_{5y} = 1,
	\end{aligned}
	\end{equation}    
	for $y \in \mathcal Y$. Using further algebraic structures, one can reduce the ODE system of $13 \kappa$ equations composed by  \eqref{eq:ode2} and \eqref{eq:w12}
	into a system of $4 \kappa$ equations of the form \eqref{eq:ode1} for the MFG characterization in Theorem \ref{t:main}.
	
	\begin{proof}
		[{\bf Proof of Theorem \ref{t:main}}]
		Since $a_y$ ($y \in \mathcal Y$) has the same expressions as \eqref{eq:ode1}, its existence, uniqueness and boundedness are shown in Lemma \ref{l:eu_abc}. Given $a_y$ ($y \in \mathcal Y$) and smooth bounded $w$'s, $$\left(b_y,d_y,e_y,f_y: y \in \mathcal Y \right)$$ is a coupled linear system, and their existence, uniqueness and boundedness is shown by Theorem 12.1 in \cite{antsaklis2006}. Similarly, given $(b_y,d_y, f_y: y \in \mathcal Y$), $(k_y, c_y: y \in \mathcal Y$) is a linear system, and their existence and uniqueness is also guaranteed by Theorem 12.1 in \cite{antsaklis2006}. 
		
		The ODE system \eqref{eq:ode2} can be rewritten by
		\begin{equation*}
			\begin{aligned}
				\begin{cases}
				\displaystyle a_y' + 2\tilde{b}_{1y}a_y - 2 \tilde{b}_{2y}^2 a_y^2 + \sum_{i=1}^{\kappa} q_{y,i} a_i + h_y = 0, \\
				\displaystyle d_y' + \tilde{b}_{1y} d_y - 2 \tilde{b}_{2y}^2 a_y d_y - \tilde{b}_{2y}^2 f_y d_y  + \sum_{i=1}^{\kappa} q_{y,i} d_i = 0 , \\
				\displaystyle e_y' -\tilde{b}_{2y}^2 d_y f_y -2 \tilde{b}_{2y}^2 k_y d_y+ e_y \left(\tilde{b}_{1y} - 2 \tilde{b}_{2y}^2 a_y - \tilde{b}_{2y}^2 f_y \right) -2 \tilde{b}_{2y}^2 b_y d_y + \sum_{i=1}^{\kappa} q_{y,i} e_i = 0,\\
				\displaystyle f_y' + \tilde{b}_{1y} f_y - 2 \tilde{b}_{2y}^2 a_y f_y + f_y \left(\tilde{b}_{1y} - 2 \tilde{b}_{2y}^2 a_y - \tilde{b}_{2y}^2 f_y \right)+ \sum_{i=1}^{\kappa} q_{y,i} f_i - 2h_y  = 0, \\
			    \displaystyle k_y' -\frac{1}{2} \tilde{b}_{2y}^2 f_y^2 + 2k_y \left(\tilde{b}_{1y} - 2 \tilde{b}_{2y}^2 a_y - \tilde{b}_{2y}^2 f_y \right) -2 \tilde{b}_{2y}^2b_y f_y + \sum_{i=1}^{\kappa} q_{y,i} k_i= 0, \\
				\displaystyle b_y' +b_y \left(-4 \tilde{b}_{2y}^2 a_y + 2 \tilde{b}_{1y} \right) + \sum_{i=1}^{\kappa} q_{y,i} b_i + h_y = 0, \\
				\displaystyle c_y' + a_y -\frac{1}{2}\tilde{b}_{2y}^2 d_y^2 -2 \tilde{b}_{2y}^2 d_y e_y + b_y + \sum_{i=1}^{\kappa} q_{y,i} c_i = 0,
			\end{cases}
			\end{aligned}
		\end{equation*}
		with the terminal conditions
		\begin{equation*}
			\begin{aligned}
				&a_y(T) = g_y, \ b_y(T) = g_y, \  c_y(T) = 0, \ d_y(T) = 0, \ e_y(T) = 0, \ f_y(T) = -2g_y, \ k_y(T) = 0.
			\end{aligned}
		\end{equation*}
		
		Since $a_y, b_y$ ($y \in \mathcal Y$) has the same expressions as \eqref{eq:ode1}, 
		its existence, uniqueness and boundedness are shown in Lemma \ref{l:eu_abc}. Meanwhile, with the given 
		$(a_y, b_y: y \in \mathcal Y)$, we denote $l_y = 2a_y+f_y$, and then
		\begin{equation*}
			l_y' + 2 \tilde{b}_{1y} l_y - \tilde{b}_{2y}^2 l_y^2 + \sum_{i=1}^{\kappa} q_{y,i} l_i = 0 \text{ , } l_y(T) = 0.
		\end{equation*}
		By Lemma \ref{l:existN} and Lemma \ref{l:bdN} in Appendix, there exists a unique solution for $l_y \ (y \in \mathcal Y)$, which is $l_y = 0, y \in \mathcal Y$. This gives $f_y = -2a_y$ and $d_y' + \tilde{b}_{1y} d_y + \sum_{i=1}^{\kappa} q_{y,i} d_i = 0$, which implies $d_y = 0, y \in \mathcal Y$. Then, the equation for $e_y$ can be simplified as $e_y' + \tilde{b}_{1y} e_y + \sum_{i=1}^{\kappa} q_{y,i} e_i = 0$, which indicates that $e_y = 0, y \in \mathcal Y$. For $k_y, c_y$, with the given of $(a_y, b_y: y \in \mathcal Y)$, we have
		\begin{equation*}
		\begin{aligned}
		 & k_y' + 2 \tilde{b}_{1y} k_y - 2 \tilde{b}_{2y}^2 a_y^2  + 4 \tilde{b}_{2y}^2 a_y b_y + \sum_{i=1}^{\kappa} q_{y,i} k_i = 0 \text{ , } k_y(T) = 0, \\
		 & c_y' + a_y + b_y + \sum_{i=1}^{\kappa} q_{y,i} c_i = 0 \text{ , } c_y(T) = 0.
		\end{aligned}
		\end{equation*}
		The existence and uniqueness of the solution for $k_y, c_y \ (y \in \mathcal Y)$ are yielded by Theorem 12.1 in \cite{antsaklis2006}.

		Note that in this case, since $2a_y + f_y = 0$ and $d_y=0$ for $y \in \mathcal Y$, from \eqref{eq:mu_sigma_gen} we have $$\hat{\mu}_s = \bar\mu + \int_t^s \tilde{b}_1(Y_r, r) \hat{\mu}_r \ d r$$
		for all $s \in [t, T]$. Then
		$$\hat\nu_s = \bar\nu + \int_t^s \left(1 + 2 \tilde{b}_1(Y_r, r) \hat{\nu}_r - 4 \tilde{b}_2^2(Y_r, r) a_{Y_r}(r) \hat{\nu}_r + 4 \tilde{b}_2^2(Y_r, r) a_{Y_r}(r) \hat{\mu}_r^2 \right) \, d r.$$
		Plugging $d_y = 0$ for $y \in \mathcal Y$ back to \eqref{eq:alpha}, we obtain the optimal control by
		$$\hat{\alpha}_s = -2 \tilde{b}_2^2(Y_s, s) a_{Y_s}(s) \left(\hat{X}_s - \hat{\mu}_s \right)
		.$$
		Since we have $d_y = 0$ for $y \in \mathcal Y$,  the value function can be simplified from  \eqref{eq:v1} to
		$$v_y(x, t,\bar\mu, \bar\nu) = a_y(t) x^2 {-2a_y(t)x\bar\mu + k_y(t)\bar\mu^2}+ b_y(t) \bar \nu + c_y(t).$$
		By the equivalence Lemma \ref{l:mkv}, it  yields the value function $U$ of Theorem \ref{t:main} .
		Moreover, 
		{since $f_y = -2a_y$ and $k_y \ne 0$,} the ODE system \eqref{eq:ode2} together with \eqref{eq:w12} can be reduced into \eqref{eq:ode1}. From the Lemma \ref{l:eu_abc}, the existence and uniqueness of $(a_y, b_y, c_y, k_y: y \in \mathcal Y)$ in \eqref{eq:ode1} is guaranteed.

	\end{proof}

	\section{The $N$-Player Game and its Convergence to MFGs}
	\label{s:section4}
	
	In this section, we show the convergence of the $N$-player game to MFGs. 
	To simplify the presentation, we may omit the superscript $(N)$ for the processes in the probability space $\Omega^{(N)}$, whenever there is no confusion. 
	First, we solve the $N$-player game in Subsection~\ref{s:raw}, which provides a Riccati system consisting of $O(N^3)$ equations. Subsection~\ref{s:simplify}
	reduces the corresponding Riccati system into an ODE system whose dimension is independent of $N$.
	This becomes the key building block of the convergence rate obtained in Subsection~\ref{s:convergence}.
	To obtain the convergence rate, Subsection \ref{s:convergence} provides an explicit embedding of some processes in $\Omega^{(N)}$ into the probability space $\Omega$. Note that, $\Omega^{(N)}$ is much richer than $\Omega$ since $\Omega^{(N)}$
	contains $N$ Brownian motions while $\Omega$ has only two Brownian motions. Therefore, careful treatment has to be carried out to some processes of our interest, otherwise, such an embedding is in general implausible.

	\subsection{Characterization of the $N$-player game by Riccati system}\label{s:raw}
	The $N$-player game is indeed an $N$-coupled stochastic LQG problem by its very own definition, see Subsection \ref{s:n-player}. Therefore, the solution can be derived via Riccati system with the existing LQG theory given below:
	For $i = 1, 2, \ldots, N$, $y \in \mathcal Y$,
	\begin{equation}
		\label{eq:ABC}
		\begin{cases}
			\displaystyle A_{iy}' + 2 \tilde{b}_{1y} e_i e_i^\top A_{iy} - 2 \tilde{b}_{2y}^2 A_{iy}^\top e_i e_i^\top A_{iy} + \sum_{j \ne i}^{N} \left(2\tilde{b}_{1y} e_j e_j^\top A_{iy} -4 \tilde{b}_{2y}^2 A_{jy}^\top e_je_j^\top A_{iy} \right) \\
			\displaystyle \hspace{1.5 in} + \sum_{j=1}^{\kappa} q_{y,j} A_{ij} + \frac{h_y}{N} \sum_{j \ne i}^{N} \left(e_i-e_j\right) \left(e_i-e_j\right)^\top= 0,\\
			\displaystyle B_{iy}' + \sum_{j \ne i}^{N} \left(\tilde{b}_{1y} e_j e_j^\top B_{iy} -2 \tilde{b}_{2y}^2 A_{iy}^\top e_je_j^\top B_{jy} - 2 \tilde{b}_{2y}^2 A_{jy}^\top e_j e_j^\top B_{iy}\right) \\
			\displaystyle \hspace{1.5 in} + \tilde{b}_{1y} e_i e_i^\top B_{iy} -2 \tilde{b}_{2y}^2 A_{iy}^\top e_i e_i^\top B_{iy} + \sum_{j=1}^{\kappa} q_{y,j} B_{ij} = 0,  \\
			\displaystyle C_{iy}' -\frac{1}{2} \tilde{b}_{2y}^2 B_{iy}^\top e_ie_i^\top B_{iy} - \sum_{j \ne i}^{N} \tilde{b}_{2y}^2 B_{jy}^\top e_je_j^\top B_{iy} + \sum_{j=1}^{N} tr(A_{jy}) + \sum_{j=1}^{\kappa} q_{y,j} C_{ij} =0,\\
			\displaystyle A_{iy}(T) = \frac{g_y}{N}\Lambda_i, \ B_{iy}(T) = 0 \cdot \mathds{1}_{N}, \ C_{iy}(T) = 0,
		\end{cases}
	\end{equation}
	where the solutions consist of 
	$N\times N$ symmetric matrices $A_{iy}$'s,
	$N$-dimensional vectors $B_{iy}$'s,
	and $C_{iy} \in \mathbb R$.  
	In the above, $\mathds{1}_{N}$ is the $N$-dimensional vector with all entries are $1$, $\Lambda_i$'s are $N \times N$ matrices with diagonal $1$ except $(\Lambda_i)_{ii} = N-1$, $\left(\Lambda_i\right)_{ij} = \left(\Lambda_i\right)_{ji} = -1$ for any $j \ne i$ and the rest entries as 0, and $e_i$'s are the $N$-dimensional natural basis. 
	
	\begin{lemma}
		\label{l:riccati-N}
		Suppose $(A_{iy}, B_{iy}, C_{iy}: i = 1, 2, \ldots, N, \ y \in \mathcal Y)$ is the solution of \eqref{eq:ABC}.
		Then,  
		the value functions of $N$-player game defined by \eqref{eq:value_i} are
		$$V_i(y, x^{(N)}) = 
		(x^{(N)})^\top A_{iy}(0) x^{(N)} + (x^{(N)})^\top B_{iy}(0) + C_{iy}(0), \quad i=1, 2, \dots, N.$$ 
		Moreover, the path and the control under the equilibrium are
		\begin{equation}
			\label{eq:Xihat}
			d \hat X_{it} = \left( \tilde{b}_1(Y_t, t) \hat X_{it} - \tilde{b}_2^2(Y_t, t) \left( 2 (A_{i  Y_t})_i^\top \hat X_t +(B_{iY_{t}})_i \right)\right) dt + dW_{it}, \quad i=1, 2, \dots, N,
		\end{equation}
		and
		$$ \hat{\alpha}_{it}   = - \tilde{b}_2(Y_t, t) \left(2 (A_{i  Y_t})_i^\top \hat X_t +(B_{iY_{t}})_i \right),$$
		where $(A)_i$ denotes the $i$-th column of matrix $A$, $(B)_i$ denotes the $i$-th entry of vector $B$ and $\hat X_t = [\hat X_{1t}, \hat X_{2t}, \dots, \hat  X_{Nt}]^{\top}$.
	\end{lemma}
	\begin{proof}
	It is standard that, under enough regularities, the value function $V (y, x^{(N)})= (V_1, V_2, \dots, V_N) (y, x^{(N)})$ of the $N$-player game can be lifted to  the solution $v_{iy}(x^{(N)}, t)$ of the following system of HJB equations,
	for $i = 1, 2, \ldots, N$ and $y \in \mathcal Y$, 
		\begin{align}
			\begin{cases}
				\displaystyle \partial_t v_{iy} + \tilde{b}_{1y} x_i \partial_i v_{iy} - \frac{1}{2}\left( \tilde{b}_{2y} \partial_i v_{iy}\right)^2 + \sum_{j \ne i}^{N} \left(\tilde{b}_{1y} x_j - \tilde{b}_{2y}^2 \partial_j v_{jy} \right) \partial_j v_{iy} \\
				\displaystyle \hspace{1.5 in} + \frac{1}{2}\Delta v_{iy} + \sum_{j=1}^{\kappa} q_{y,j} v_{ij}
				+ \frac{h_y}{N} \sum_{j \ne i}^{N} \left(\left(e_i-e_j\right)^\top x^{(N)}\right)^2= 0, \\
				\displaystyle v_{iy}(x^{(N)},T) = \frac{g_y}{N}\sum_{j \ne i}^{N} \left((e_i-e_j)^\top x^{(N)}\right)^2.
			\end{cases}
			\label{V}
		\end{align}
		Then, 
		the value functions $V$ of $N$-player game defined by \eqref{eq:value_i} is
		$V_i(y, x^{(N)}) = v_{iy}(x^{(N)}, 0)$ for all $i=1, 2, \dots, N$. Moreover, the path and the control under the equilibrium are
		$$d \hat X_{it} = \left(\tilde{b}_1 (Y_t, t) \hat X_{it} -\tilde{b}_2^2 (Y_t, t) \partial_i v_{i Y_t} (\hat X_t, t) \right) dt + dW_{it},\quad  i=1, 2, \dots, N,$$
		and
		$$ \hat{\alpha}_{it}   = -\tilde{b}_2 (Y_t, t) \partial_i v_{i Y_t} (\hat X_t, t).$$
		The proof is the application of Dynkin's formula and the details are  omitted here.
		Due to its LQG structure, the value function leads to a quadratic function of the form
		$$v_{iy}(x^{(N)}, t) = (x^{(N)})^\top A_{iy}(t) x^{(N)} + (x^{(N)})^\top B_{iy}(t) + C_{iy}(t).$$ 
		For each $i = 1, 2, \dots, N$, after plugging $V_{iy}$ into \eqref{V}, and matching the coefficient of variables, we get the desired results.
	\end{proof}

	\subsection{Reduced Riccati form for the equilibrium}\label{s:simplify}
	So far, the $N$-player game and MFG have been characterized by Lemma \ref{l:riccati-N} and Theorem \ref{t:main}, respectively.
	One of our main objectives is to investigate the convergence of the generic 
	optimal path $\hat X_{1t}^{(N)}$ of $N$-player game generated 
	\eqref{eq:ABC}-\eqref{eq:Xihat} 
	  to the optimal path 
	$\hat X_t$ of MFG generated by \eqref{eq:ode1}-\eqref{eq:Xhat03}. 
	
	Note that $\hat X_t$  relies only on $\kappa$ functions
	$(a_y: y \in \mathcal Y)$ from the simple ODE system \eqref{eq:ode1} 
	while $\rho(\hat X_t^{(N)})$ depends on $O(N^3)$ functions from
	$(A_{iy}: i= 1, 2, \ldots, N, \  y \in \mathcal Y)$ solved from a huge Riccati system \eqref{eq:ABC}. 
	Therefore, it is almost a hopeless task for a meaningful comparison between these two processes 
	without gaining further insight into the complex structure of the Riccati system \eqref{eq:ABC}.
	
	To proceed, let us first observe some hidden patterns from a numerical result for the solution of Riccati \eqref{eq:ABC}.
	The following matrix shows $A_{20}$ at $t = 1$ for $N = 5$ with the same parameters as in 
	Figure \ref{fig: figure 1} and Figure \ref{fig: figure 2} in Section \ref{s:simu}:
\begin{equation*}
\label{eq:table}
A_{20} (1) = 
\begin{bmatrix}
0.1319	& -0.1924 &	0.0202	& 0.0202 & 0.0202\\
			-0.1924 & 0.7696 & -0.1924 & -0.1924 & -0.1924\\
			0.0202 &-0.1924 & 0.1319 & 0.0202 &	0.0202 \\
			0.0202 & -0.1924 & 0.0202 & 0.1319 & 0.0202 \\
			0.0202 & -0.1924 &	0.0202 & 0.0202 & 0.1319
\end{bmatrix}.
\end{equation*}

	Interestingly enough, we observe that the entire 25 entries of $A_{20}(1)$ indeed consists of $4$ distinct values. Moreover, similar computation with different values of $N$ only yields a larger table depending on $N$, but always consists of $4$ values.
	Inspired by this accidental discovery from the above numerical example, we may want to believe and prove a pattern of the matrix $A_{iy}$ in the following form:
	\begin{equation}\label{eq:magic}
		(A_{iy})_{pq} = \begin{cases}
			a_{1y}(t), & \text{ if } p=q=i,\\
			a_{2y}(t), & \text{ if } p=q\ne i,\\
			a_{3y}(t), & \text{ if } p\ne q, p = i \text{ or } q = i, \\
			a_{4y}(t), & \text{ otherwise},
		\end{cases}
	\end{equation}
	for $y \in \mathcal Y$. The next result justifies the above pattern: 
	the $N^2$ entries of the matrix $A_{iy}$ can be embedded to a $2 \kappa$-dimensional vector space no matter how big $N$ is.

	\begin{lemma}
		\label{l:ABCexist}
		There exists a unique solution $(a_{1y}^{N} , a_{2y}^{N} )$ from the ODE system\eqref{eq:a_12}
		\begin{equation}
			\label{eq:a_12}
			\begin{cases}
				\displaystyle a_{1y}' + 2 \tilde{b}_{1y} a_{1y} - \frac{2(N+1)}{N-1} \tilde{b}_{2y}^2 a_{1y}^2 + \sum_{j=1}^{\kappa} q_{y,j} a_{1j} + \frac{N-1}{N} h_y = 0, \\
				\displaystyle a_{2y}' + 2 \tilde{b}_{1y} a_{2y} +\frac{2}{(N-1)^2} \tilde{b}_{2y}^2 a_{1y}^2 -\frac{4N}{N-1} \tilde{b}_{2y}^2 a_{1y} a_{2y} + \sum_{j=1}^{\kappa} q_{y,j} a_{2j} + \frac{h_y}{N} = 0,\\
				\displaystyle a_{1y}(T) = \frac{N-1}{N} g_y, \ a_{2y}(T) = \frac{g_y}{N},
			\end{cases}
		\end{equation}
		for $y \in \mathcal Y$.
		Moreover, the path and the control of player $i$ under the equilibrium are
		\begin{equation}
			\label{eq:XihatN}
			d \hat X_{it}^{(N)}  = \left( \tilde{b}_1(Y_t^{(N)}, t) \hat X_{it}^{(N)} -2 \tilde{b}_2^2(Y_t^{(N)}, t) a_{1Y_t^{(N)}}^{N} (t) \left( 
		\hat X_{it}^{(N)} -\frac{1}{N-1}\sum_{j \ne i}^N  \hat X_{jt}^{(N)}
		\right) \right) dt + dW_{it}^{(N)},
		\end{equation}
		and
		$$ \hat{\alpha}_{it}^{(N)} = -2 \tilde{b}_2(Y_t^{(N)}, t) a_{1Y_t^{(N)}}^{N} (t) \left( 
		\hat X_{it}^{(N)} -\frac{1}{N-1}\sum_{j \ne i}^N  \hat X_{jt}^{(N)}
		\right)$$
		for $i=1, 2, \dots, N$.
	\end{lemma}
	\begin{proof}
		It is obvious to see that in the Riccati system \eqref{eq:ABC}, $B_{iy} = 0$ for all $i = 1, 2, \dots, N$ and $y \in \mathcal Y$. Note that in this case, for $i = 1, 2, \dots, N$, the  optimal control is given by
		\begin{align}
			\hat{\alpha}_i^{(N)}  = - 2 \tilde{b}_2 (Y_t^{(N)}, t) \sum_{j =1}^N (A_{i  Y_t^{(N)}})_{ij} \hat X_{jt}^{(N)} = -2 \tilde{b}_2 (Y_t^{(N)}, t) \left(A_{i  Y_t^{(N)}}\right)_i^\top \hat X^{(N)}_t. \notag
		\end{align}
		
		Plugging the pattern \eqref{eq:magic} into the differential equation of $A_{iy}$, we have
		\begin{equation*}
			\begin{aligned}
				& a_{1y}' + 2\tilde{b}_{1y} a_{1y} - 2 \tilde{b}_{2y}^2 a_{1y}^2 - 4(N-1) \tilde{b}_{2y}^2 a_{3y}^2 + \sum_{j=1}^{\kappa} q_{y,j} a_{1j} + \frac{N-1}{N} h_y = 0, \notag \\
				& a_{2y}' + 2 \tilde{b}_{1y} a_{2y} - 2 \tilde{b}_{2y}^2 a_{3y}^2 - 4 \tilde{b}_{2y}^2 \left( a_{1y}a_{2y} + (N-2)a_{3y}a_{4y} \right) + \sum_{j=1}^{\kappa} q_{y,j} a_{2j} + \frac{h_y}{N} = 0,\notag \\
				& a_{3y}' + 2 \tilde{b}_{1y} a_{3y} - 2 \tilde{b}_{2y}^2 a_{1y}a_{3y} - 4  \tilde{b}_{2y}^2 \left( a_{1y}a_{3y}+ (N-2)a_{3y}^2 \right) + \sum_{j=1}^{\kappa} q_{y,j} a_{3j} - \frac{h_y}{N} = 0, \notag\\
				& a_{3y}' + 2 \tilde{b}_{1y} a_{3y} - 2 \tilde{b}_{2y}^2 a_{1y}a_{3y} - 4  \tilde{b}_{2y}^2 \left( a_{2y}a_{3y}+ (N-2)a_{3y} a_{4y} \right) + \sum_{j=1}^{\kappa} q_{y,j} a_{3j} - \frac{h_y}{N} = 0, \notag\\
				& a_{4y}' + 2 \tilde{b}_{1y} a_{4y} -2 \tilde{b}_{2y}^2 a_{3y}^2 -4 \tilde{b}_{2y}^2 \left(a_{2y}a_{3y} + a_{1y}a_{4y} + (N-3)a_{3y}a_{4y} \right) + \sum_{j=1}^{\kappa} q_{y,j} a_{4j} = 0,   
			\end{aligned}
		\end{equation*}
		which gives $a_{1y} +(N-2)a_{3y} = a_{2y} + (N-2) a_{4y}$ since two expressions for $a_{3y}$ should be identical. 
		This implies that $\left(a_{1y} +(N-2)a_{3y}\right)' =\left( a_{2y} + (N-2) a_{4y}\right)'$ or
		\begin{equation*}
			\begin{aligned}
		       &-2\tilde{b}_{1y} a_{1y} + 2\tilde{b}_{2y}^2 a_{1y}^2+ 4(N-1)\tilde{b}_{2y}^2 a_{3y}^2 -\frac{N-1}{N} h_y - \sum_{j=1}^{\kappa} q_{y,j} a_{1j} \\
			    &+(N-2) \left(-2 \tilde{b}_{1y} a_{3y} + 2 \tilde{b}_{2y}^2 a_{1y} a_{3y} + 4 \tilde{b}_{2y}^2 \left(a_{2y} a_{3y} + (N-2) a_{3y} a_{4y} \right) - \sum_{j=1}^{\kappa} q_{y,j} a_{3j} + \frac{h_y}{N} \right) \\
				=& -2\tilde{b}_{1y} a_{2y} + 2\tilde{b}_{2y}^2 a_{3y}^2 + 4\tilde{b}_{2y}^2 \left(a_{1y} a_{2y} +(N-2) a_{3y} a_{4y} \right) - \sum_{j=1}^{\kappa} q_{y,j} a_{2j} - \frac{h_y}{N} \\
				&+(N-2) \left(-2 \tilde{b}_{1y} a_{4y} + 2 \tilde{b}_{2y}^2 a_{3y}^2 + 4 \tilde{b}_{2y}^2 \left(a_{1y} a_{4y} + a_{2y} a_{3y} + (N-3) a_{3y} a_{4y} \right) - \sum_{j=1}^{\kappa} q_{y,j} a_{4j} \right).
			\end{aligned}
		\end{equation*}
		After combining terms and substituting $a_{2y}+(N-2)a_{4y}$ with $a_{1y} +(N-2) a_{3y}$, we get $a_{1y}^2 +(N-2)a_{1y}a_{3y} - (N-1)a_{3y}^2 = 0$, which yields $a_{3y} = a_{1y}$ or $a_{3y} = -\frac{1}{N-1} a_{1y}$. Note that $a_{3y}\ne a_{1y}$ due to their different differential equations. Hence, we can conclude that $a_{3y} = -\frac{1}{N-1} a_{1y}$. In conclusion, for $i = 1, 2, \dots, N$, $A_{iy}$ ($y \in \mathcal Y$) has the following expressions:
		\begin{equation*}
			(A_{iy})_{pq} =
			\begin{cases}
				a_{1y}(t), & \text{ if } p=q=i,\\
				a_{2y}(t), & \text{ if } p=q\ne i,\\
				-\frac{1}{N-1}a_{1y}(t), & \text{ if } p\ne q, p = i \text{ or } q = i,\\
				\frac{1}{(N-1)(N-2)}a_{1y}(t)- \frac{1}{N-2}a_{2y} (t), & \text{ otherwise}.
			\end{cases}
		\end{equation*}
		The existence and uniqueness of \eqref{eq:ABC} is equivalent to the existence and uniqueness of \eqref{eq:a_12}. For $a_{1y}$, the existence and uniqueness can be deduced from Lemma \ref{l:existN} and \ref{l:bdN}. Given $a_{1y}$'s, $a_{2y}$'s are linear equations, thus their existence and uniqueness are guaranteed by Theorem 12.1 in \cite{antsaklis2006}. Together with previous discussions, we conclude the results.
	\end{proof}

	\subsection{Convergence}\label{s:convergence}
	Based on the current progress, let us reiterate our goal (P1) for the convergence.
	Our objective is the convergence of the joint distribution $\mathcal L(\hat X_{1t}^{(N)}, Y_t^{(N)})$ 
	of $N$-player game generated 
	\eqref{eq:a_12}-\eqref{eq:XihatN} in the probability space $\Omega^{(N)}$
	  to the distribution 
	$\mathcal L (\hat X_t, Y_t) $ of MFG generated by \eqref{eq:ode1}-\eqref{eq:Xhat03} in the probability space $\Omega$.  More precisely, we want to find a number $\eta>0$ satisfying 
	\begin{equation}
	\label{eq:W2}
	\mathbb W_2 \left(\mathcal L(\hat X_{1t}^{(N)}, Y_t^{(N)}), \mathcal L (\hat X_t, Y_t) \right) = O \left(N^{-\eta} \right),
	\end{equation}
	where $\mathbb W_2$ is the 2-Wasserstein metric.
	This procedure is given in the following two steps:
	\begin{enumerate}
	\item
	 We will construct a process $Z^N$ in the probability space $\Omega$, who provides
	 exact copy of the joint distribution in the sense of 
	  $$\mathcal L(\hat X_{1t}^{(N)}, Y_t^{(N)}) = \mathcal L(Z^N, Y).$$
	 Note that, the \eqref{eq:XihatN} shows that $\hat X_{1t}^{(N)}$ correlates 
	 to $N$ many Brownian motions $\{W_i^{(N)}: i = 1, 2, \ldots, N\}$
	 from a much richer space $\Omega^{(N)}$ while $\Omega$ is a much 
	 smaller space   having only two Brownian motions $W$ and $B$.
	 Therefore, such an embedding essentially requires to represent $\hat X_{1t}^{(N)}$ by two independent Brownian motions 
	 and is in general not possible. 
	 However, due to the symmetric structure of MFG (or the nature of the mean field effect),
	 the embedding is possible and the details are provided in 
	 Lemma \ref{eq:coupling}.
	 \item By Proposition \ref{p:conv}, we can use distribution copy $(Z^N, Y)$ in $\Omega$ to write
	 \begin{equation}
	 \label{eq:w22}
	 \mathbb W_2^2 \left(\mathcal L(\hat X_{1t}^{(N)}, Y_t^{(N)}), \mathcal L (\hat X_t, Y_t) \right) 
	 \le \mathbb  E \left[ \left\vert Z_t^N - \hat X_t \right\vert^2 \right].
	 \end{equation}
	 To obtain the estimate of the above right hand side, we shall compare the \eqref{eq:ZN} 
	 of $Z^N$ and \eqref{eq:Xhat03} of $\hat X$, and 
	  it becomes essential to obtain the convergence rate of the ODE system \eqref{eq:a_12} towards the ODE system 
	 \eqref{eq:ode1}. The details are provided in Lemma \ref{l:compare}.
	\end{enumerate}

	\begin{lemma}
		\label{eq:coupling}
		Let $\{X_0^i: i\in \mathbb N\}$ be i.i.d. random variables in $\Omega$ independent to $(W, B, Y)$ 
		with $X_0^1 = X_0$.
		Let $Z^N$ be the solution of
		\begin{equation}\label{eq:ZN}
			Z_{t}^{N} = X_0 + \int_0^t \tilde{b}_1(Y_s, s) Z_s^{N} ds - 
			\int_0^t 2 \tilde{b}_2^2(Y_s, s) \hat a_{1Y_s}^{N} (s)
			\left(Z_{s}^N - \bar{X}_s^N
			\right) ds + W_{t},
		\end{equation}
		where 
		$$d \bar{X}_t^N = \tilde{b}_1(Y_t, t) \bar{X}_t^N d t + \frac{\sqrt{N-1}}{N} d B_t + \frac{1}{N} d W_t, 
		\quad \bar{X}_0^N = \frac 1 N \sum_{i=1}^N X_0^i,$$ 
		and 
		$$\hat a_{1y}^{N} =  \frac{N}{N-1} a_{1y}^{N},$$
		where $a_{1y}^{N}$ is from the ODE system\eqref{eq:a_12}.
		Then, $(Z^N_t, Y_t)$ in $(\Omega, \mathcal F_T, \mathbb P)$  has the same distribution 
		as $(\hat X_{1t}^{(N)}, Y_t^{(N)})$ in $(\Omega^{(N)}, \mathcal F_T^{(N)}, \mathbb P^{(N)})$.
	\end{lemma}
	\begin{proof}
		Continued from the Lemma \ref{l:ABCexist}, player $i$'s path in the $N$-player game follows
		\begin{equation*}
			\label{eq:X111}
			\hat X_{it}^{(N)} = x_i^{(N)} + \int_0^t \tilde{b}_1(Y_s^{(N)}, s) \hat X_{is}^{(N)} ds - \int_0^t 2 \tilde{b}_2^2(Y_s^{(N)}, s) a_{1Y_s^{(N)}}^{N} (s)
			\left( 
			\hat X_{is}^{(N)} -\frac{1}{N-1}\sum_{j \ne i}^N  \hat X_{js}^{(N)}
			\right) ds + W_{it}^{(N)}.
		\end{equation*}
		With the notation
		$$\bar X^{(N)}_s = \frac{1}{N} \sum_{i=1}^N \hat  X_{is}^{(N)},
		$$
		one can rewrite the path by
		\begin{equation}
			\label{eq:X_Ni}
			\hat X_{it}^{(N)} = x_i^{(N)} + \int_0^t \tilde{b}_1(Y_s^{(N)}, s) \hat X_{is}^{(N)} ds - \int_0^t 2 \tilde{b}_2^2(Y_s^{(N)}, s) \hat{a}_{1Y_s^{(N)}}^{N} (s)
			\left( 
			\hat X_{is}^{(N)} - \bar{X}_s^{(N)}
			\right) ds + W_{it}^{(N)}.
		\end{equation}
		By adding up the above equations \eqref{eq:X_Ni} indexed by $i=1$ to $N$, one can have
		\begin{equation}
			\label{eq:X_bar}
			\begin{aligned}
			\bar X^{(N)}_t &= \bar x^{(N)} + \int_0^t \tilde{b}_1(Y_s^{(N)}, s) \bar{X}_s^{(N)} ds + \frac 1 N \sum_{i=1}^N W_{it}^{(N)}  \\
			&= \bar x^{(N)} + \int_0^t \tilde{b}_1(Y_s^{(N)}, s) \bar{X}_s^{(N)} ds 
			+ \frac  {\sqrt {N-1}} {N} \left(\sqrt{N-1} \bar W_{-it}^{(N)}\right) + \frac 1 N W_{it}^{(N)},
			\end{aligned}
		\end{equation}
		where $\bar W_{-it}^{(N)} :=  \frac 1 {N-1}\sum_{j\neq i} W_{jt}^{(N)}$.

		Next, 
		we define
		solution maps 
		of \eqref{eq:X_Ni} and \eqref{eq:X_bar}:
		\begin{equation}
		\label{eq:Gbar}
		\bar G_t(x, \phi, W_1, W_2) = \mathcal E_t(\phi) 
		\left(x + 
		\int_0^t \mathcal E_s(- \phi) d (W_{1s} + W_{2s})
		\right)
		\end{equation}
		and
			\begin{equation}
	\label{eq:function G}
	G_t(x, \phi_1, \phi_2, \phi_3, W) = x\mathcal E_t(\phi_1 - \phi_2) + \mathcal E_t(\phi_1 - \phi_2) \int_0^t \mathcal E_s (-\phi_1 + \phi_2) \left(\phi_2(s) \phi_3(s) ds + dW_s \right),
	\end{equation}
where
$$\mathcal E_t(\phi)
		= \exp \left\{ \int_0^t \phi_s ds \right\}.
		$$
		Now, we can rewrite $\bar X_t^{(N)}$ of \eqref{eq:X_bar} and $\hat X_{1t}^{(N)}$ of  \eqref{eq:X_Ni} as
		$$\bar X_t^{(N)} = \bar G_t 
		\left (\frac 1 N \sum_{i=1}^N x_i^{(N)}, \tilde b_1(Y_.^{(N)}, \cdot),  
		\frac  {\sqrt {N-1}} {N} \left(\sqrt{N-1} \bar W_{-1}^{(N)}\right), \frac 1 N W_{1}^{(N)}\right ), $$
		and 
		$$\hat X_{1t}^{(N)} 
		= 
		G_t \left(x_1^{(N)}, \tilde{b}_1(Y^{(N)}_{\cdot}, \cdot), 2 \tilde{b}_2(Y^{(N)}_{\cdot}, \cdot) \hat a^N_1(Y^{(N)}_{\cdot}, \cdot), 
		\bar {X}^{(N)} (\cdot), W_1^{(N)} \right)$$
		Meanwhile, $(Z^N, \bar X^N)$ of \eqref{eq:ZN} can also be written in the form of
		$$\bar X_t^{N} = \bar G_t 
		\left (\frac 1 N \sum_{i=1}^N X_0^i, \tilde b_1(Y_., \cdot), 
		\frac  {\sqrt {N-1}} {N} B, \frac 1 N W \right ), $$
		and 
		\begin{equation}
		\label{eq:ZN1}
		Z_t^N = G_t \left(X_0, \tilde{b}_1(Y_{\cdot}, \cdot), 2 \tilde{b}_2(Y_{\cdot}, \cdot) 
		\hat a^N_1(Y_{\cdot}, \cdot), \bar{X}^{N}(\cdot), W \right)
		\end{equation}
		Finally, the fact that the distribution of $(Z^N, Y)$  in the space $\Omega$ is
		identical distribution to $(\hat X_{1}^{(N)}, Y^{(N)})$ in $\Omega^{(N)}$ comes from the followings:
		\begin{itemize}
		\item $\tilde b_1, \tilde b_2, \hat a_1^N$ are deterministic functions.
		\item The random processes $(\sqrt{N-1} \bar W_{-1}^{(N)}, W_1^{(N)}, Y^{(N)})$ 
		are independent mutually in $\Omega^{(N)}$, while the random elements $(B, W, Y)$ are also independent triples. 
		Moreover, two random triples have identical joint distributions. 
		\item Initial states are generated from identical joint distributions $\{x_i^{(N)}: i = 1, 2, \dots, N\}$ and 
		$\{X_0^i: i = 1, 2, \dots, N\}$.
		\end{itemize}
		Therefore, $(Z^N, Y)$  and
		 $(\hat X_{1}^{(N)}, Y^{(N)})$ have the same distributions. This completes the proof.
			\end{proof}
	
In view of \eqref{eq:w22}, we shall estimate the second moment $\mathbb E \left[ \left\vert Z_t^N - \hat X_t \right\vert^2 \right]$.
First, we can rewrite 
 $\hat X$ of \eqref{eq:Xhat03} using above representations via $G_t$:
		\begin{equation*}
			\label{eq:hatXe}
			\hat X_t = G_t \left(X_0, \tilde{b}_1(Y_{\cdot}, \cdot), 2 \tilde{b}_2(Y_{\cdot}, \cdot) a(Y_{\cdot}, \cdot), \hat{\mu}(\cdot), W \right),
		\end{equation*}
which leads to a better comparison with $Z^N$ in the form of \eqref{eq:ZN1}.
To proceed, the following properties of $G_t$ are useful for the estimate of the second moment, 
whose proof is relegated to the Appendix \ref{s:appendix lipschitz}.
Throughout the proof of the next lemma, we will use $K$ in various places as a generic constant which varies from line to line.
	
	\begin{lemma}\label{l:compare}
	    The convergence rate  under the Wasserstein metric $\mathbb W_2(\cdot, \cdot)$ is 
	    $$\mathbb W_2 \left(\mathcal L(\hat X_{1t}^{(N)}, Y_t^{(N)}), \mathcal L (\hat X_t, Y_t) \right) = 
	    O \left(N^{- \frac 1 2} \right).$$
	\end{lemma}
	\begin{proof}
	    In view of \eqref{eq:w22}, we start with
	   \begin{equation*}
	   \begin{aligned}
	       & \mathbb W_2^2 \left(\mathcal L(\hat X_{1t}^{(N)}, Y_t^{(N)}), \mathcal L (\hat X_t, Y_t) \right) 
	 \le \mathbb E \left[ \left\vert Z_t^N - \hat X_t \right\vert^2 \right]
 \\
	       =& \mathbb E \left[\left\vert 
	       G_t \left(X_0, \tilde{b}_1(Y_{\cdot}, \cdot), 2 \tilde{b}_2(Y_{\cdot}, \cdot) 
		\hat a^N_1(Y_{\cdot}, \cdot), \bar{X}^{N}(\cdot), W \right)
	       - 
	       G_t \left(X_0, \tilde{b}_1(Y_{\cdot}, \cdot), 2 \tilde{b}_2(Y_{\cdot}, \cdot) a(Y_{\cdot}, \cdot), \hat{\mu}(\cdot), W \right)\right\vert^2 \right]
	       \\
	       := & \mathbb E \left[ \left | I_1(t) - I_2(t)  \right| \right].
	   \end{aligned}
	   \end{equation*}
	   Applying the Lipschitz continuity of  $(\phi_2, \phi_3) \mapsto G_t(x, \phi_1, \phi_2, \phi_3, W)$ by 
	   Appendix \ref{s:appendix lipschitz} on the conditional expectation $\mathbb E \left[ \left | I_1(t) - I_2(t)  \right| \Big\vert  Y \right] $, we have
	   \begin{equation*}
	   \begin{aligned}
	       \mathbb E |Z_t^N - \hat X_t|^2
	      &\leq K 
	      \mathbb E \left[ \sup_{0 \leq t \leq T} \left(2 \tilde{b}_2(Y_t, t) \hat{a}_{1Y_t}^{N}(t) - 2 \tilde{b}_2(Y_t, t) a_{Y_t} (t) \right)^2 
	      + \sup_{0 \leq t \leq T} \left(\bar{X}^N(t) - \hat{\mu}(t) \right)^2  \right] \\
	       & \leq K 
	       \mathbb E \left[
	        \sup_{0 \leq t \leq T}  \left\vert \tilde{b}_2(Y_t, t) \right\vert^2  
	        \sup_{0 \leq t \leq T} \left\vert \hat{a}_{1Y_t}^{N}(t) - a_{Y_t} (t)\right\vert^2 
	        +  \sup_{0 \leq t \leq T} \left\vert \bar{X}^N(t) - \hat{\mu}(t)  \right\vert^2 \right] \\
	        & \leq K 
	       \mathbb E \left[
	        \sup_{0 \leq t \leq T} \left\vert \hat{a}_{1Y_t}^{N}(t) - a_{Y_t} (t)\right\vert^2 
	        +  \sup_{0 \leq t \leq T} \left\vert \bar{X}^N(t) - \hat{\mu}(t)  \right\vert^2 \right] 
	        	   \end{aligned}
	   \end{equation*}
	   From the dynamic of $\bar{X}^{N}$ and $\hat{\mu}$, 
	   \begin{equation*}
	   \begin{cases}
	       \displaystyle d \left(\bar{X}_t^{N} - \hat{\mu}_t \right) = \tilde{b}_1(Y_t, t) \left(\bar{X}_t^{N} - \hat{\mu}_t \right) dt + \frac{\sqrt{N-1}}{N} dB_t + \frac{1}{N} d W_t, \\
	       \displaystyle \bar{X}_0^{N} - \hat{\mu}_0 
	       = \frac 1 N \sum_{i=1}^N X_0^i - \hat{\mu}_0,
  	   \end{cases}
	   \end{equation*}
	   which can be written in terms of $\bar G_t$ of \eqref{eq:Gbar}:
	   $$
	   \bar{X}^{N}(t) - \hat{\mu}(t) = \bar G_t\left (\frac 1 N \sum_{i=1}^N X_0^i - \hat \mu_0, \tilde b_1(Y_., \cdot), 
		\frac  {\sqrt {N-1}} {N} B, \frac 1 N W \right ).
	   $$
	   Using the fact of $\left\vert \tilde b_{1y} \right\vert_\infty <\infty$ and Ito's isometry, this yields the following estimation:
	   $$
	   \mathbb E \left[  \sup_{0 \leq t \leq T} 
	   \left\vert \bar{X}^N(t) - \hat{\mu}(t)  \right\vert^2 \right] 
	   \le K \left(\frac 1 N + \mathbb E  \left\vert \frac 1 N \sum_{i=1}^N X_0^i - \hat \mu_0 \right\vert^2 \right).
	   $$
	   Note that, by central limit theorem, we have
	   $$ N \mathbb E \left[ \left\vert \frac 1 N \sum_{i=1}^N X_0^i - \hat \mu_0 \right\vert^2 \right]
	   = \mathbb E \left[ \left\vert   \frac {\sum_{i=1}^N (X_0^i - \hat \mu_0) } {\sqrt N} \right\vert^2 \right] \to Var(X_0^1) <\infty,\quad  N\to \infty,
	   $$
	   and we conclude that
	   \begin{equation}
	   \label{eq:rate1}
	    \mathbb E \left[  \sup_{0 \leq t \leq T} 
	   \left\vert \bar{X}^N(t) - \hat{\mu}(t)  \right\vert^2 \right] 
	   = O(N^{-1}).
	   \end{equation} 
	   
	   Next we investigate the boundness of
	   $$\sup_{0 \leq t \leq T} \left\vert  \hat{a}_{1Y_t}^{N}(t) - a_{Y_t} (t) \right\vert^2.$$
	   From \eqref{eq:a_12} and $\hat a_{1y}^{N} =  \frac{N}{N-1} a_{1y}^{N}$, we have
	   \begin{equation*}
	   \begin{cases}
	       \displaystyle \left(\hat{a}^{N}_{1y} \right)' + 2 \tilde{b}_{1y} \hat{a}_{1y}^{N} - \frac{2(N+1)}{N} \tilde{b}_{2y}^2 \left( \hat{a}_{1y}^N \right)^2 + \sum_{i=1}^{\kappa} q_{y,i} \hat{a}_{1i}^N + h_y = 0 \\
	       \displaystyle \hat{a}^{N}_{1y}(T) = g_y.
	   \end{cases}
	   \end{equation*}
	   Define $u_y = a_y - \hat{a}^{N}_{1y}$, let $\tau = T-t$ and denote $u_y(\tau):= u_y(T-t)$, we have
	   \begin{equation}
	   \label{eq:u diff}
	   \begin{cases}
	       \displaystyle u_y'(\tau) = 2 \tilde{b}_{1y}(\tau) u_{y}(\tau) - 2 \tilde{b}_{2y}^2(\tau) \left(a_y(\tau) + \hat{a}^N_{1y}(\tau) \right) u_y(\tau) + \frac{2}{N} \tilde{b}_{2y}^2(\tau) \left(\hat{a}_{1y}^{N}(\tau) \right)^2 + \sum_{i=1}^{\kappa} q_{y,i} u_i(\tau) \\
	       u_y(0) = 0,
	   \end{cases}
	   \end{equation}
	   which gives that
	   \begin{equation*}
	       u_y(\tau) = \int_0^{\tau} \left( 2 \tilde{b}_{1y}(s) u_{y}(s) - 2 \tilde{b}_{2y}^2(s) \left(a_y(s) + \hat{a}^N_{1y}(s) \right) u_y(s) + \frac{2}{N} \tilde{b}_{2y}^2(s) \left(\hat{a}_{1y}^{N}(s) \right)^2 + \sum_{i=1}^{\kappa} q_{y,i} u_i(s) \right) ds.
	   \end{equation*}
	   Thus for $\tau \in [0, T]$,
	   \begin{equation*}
	   \begin{aligned}
	       |u_y(\tau)| & \le \int_0^{\tau} \left( 2 \left\vert \tilde{b}_{1y} \right\vert_{\infty} |u_{y}(s)| + 2 \left\vert\tilde{b}_{2y} \right\vert_{\infty}^2 \left(|a_y|_{\infty} + \left\vert\hat{a}^N_{1y}\right \vert_{\infty} \right) |u_y(s)|  \right. \\ 
	       & \hspace{0.5 in} \left. + \frac{2}{N} \left\vert \tilde{b}_{2y} \right\vert_{\infty}^2 \left\vert\hat{a}_{1y}^{N} \right\vert_{\infty}^2 + \sum_{i=1}^{\kappa} |q_{y,i}| |u_i(s)| \right) ds.
	   \end{aligned}
	   \end{equation*}
	   Let $\left(\left\vert \tilde{b}_{1y} \right\vert_{\infty}, \left\vert \tilde{b}_{2y} \right\vert_{\infty}, |a_y|_{\infty}, \left\vert\hat{a}_{1y}^{N} \right\vert_{\infty}, \sup_{i \in \mathcal Y} |q_{y,i}| \right) \leq K_1$, then
	   $$|u_y(\tau)| \leq \frac{2}{N}  K_1^4 T + \int_0^{\tau} \left( \left(2K_1 + 4K_1^3 \right) |u_y(s)| + K_1 \sum_{i=1}^{\kappa} |u_i(s)|\right) ds.$$
	   By adding up the above equation indexed by $y=1$ to $\kappa$, one can have
	   $$\sum_{y=1}^{\kappa} |u_y(\tau)| \leq \frac{2 \kappa K_1^4 T}{N} +  \left(2K_1 + 4 K_1^3 + \kappa K_1 \right) \int_0^{\tau} \sum_{y=1}^{\kappa} |u_y(s)| d s.$$
	   Let $K_2 = 2 \kappa K_1^4 T$ and $K_3 = 2K_1 + 4 K_1^3 + \kappa K_1$, by the Gr\"onwall's inequality,
	   $$\sum_{y=1}^{\kappa} |u_y(\tau)| \leq \frac{K_2}{N} e^{K_3 \tau} \leq \frac{K_2}{N} e^{K_3 T}, \quad \forall \tau \in [0, T],$$
	   which implies that
	   $$\sum_{y=1}^{\kappa} |u_y(\tau)| \leq \frac{K}{N}, \quad \forall \tau \in [0, T].$$
	   Thus, we have
	   \begin{equation}
	   \label{eq:rate2}
	   \sup_{0 \leq t \leq T} \left\vert \hat{a}_{1Y_t}^{N}(t) - a_{Y_t} (t)  \right\vert^2 \leq \frac{K}{N^2}, \hbox{ almsot surely }.	   \end{equation}
	   Therefore, the convergence is obtained from \eqref{eq:rate1} and \eqref{eq:rate2}:
	   \begin{equation*}
	   \begin{aligned}
	      \mathbb W_2^2 \left(\mathcal L(Z_t^N), \mathcal L(\hat{X}_t) \right)
	      \leq & K 
	       \mathbb E \left[
	        \sup_{0 \leq t \leq T} \left\vert \hat{a}_{1Y_t}^{N}(t) - a_{Y_t} (t)\right\vert^2 
	        +  \sup_{0 \leq t \leq T} \left\vert \bar{X}^N(t) - \hat{\mu}(t)  \right\vert^2 \right]  = O(N^{-1}).
	   \end{aligned}
	   \end{equation*}

	\end{proof}

	\section{Numerical results}
	\label{s:section6}
	
	\subsection{Simulations of Riccati system, the value function and optimal control of the generic palyer}
	\label{s:simu}
	We have derived a $4 \kappa$ dimensional Riccati ODE system 
	\eqref{eq:ode1} to determine the parameter functions 
\begin{center}
    $(a_y, b_y, c_y, k_y: y \in \mathcal Y)$
\end{center} 
	needed for the characterization of the equilibrium and the value function. Meanwhile, we also show the solvability of the Riccati ODE system in Section \ref{s:section3}. 
	
	As mentioned earlier, different from the MFG characterization with the common noise, the derived Riccati system is essentially finite-dimensional. In this subsection, we present a numerical experiment and show some numerical results for solving the Riccati system
	to demonstrate its computational advantages.
	
	For the illustration purpose, assume the finite time horizon is given with $T =5$ and that the coefficients of the dynamic equation are listed below:
	\begin{align*}
	    & \mathcal Y = \{0, 1\}, \\
	    & Q = \left[\begin{array}{cc}
			- 0.5 & 0.5 \\
			0.6 & - 0.6
		\end{array}\right], \\
		& \tilde{b}_1(\cdot, \cdot) = 0, \ \tilde{b}_2(\cdot, \cdot) = 1,  \\
		& h_0 = 2, \ h_1 = 5, \ g_0 = 3, \ g_1 = 1,\\
		& \mu_0 = 0, \nu_0 = 2.
	\end{align*}
	Firstly, using the forward Euler’s method with the step size $\delta = 10^{-2}$, we can obtain trajectories of $(a_y, b_y, c_y: y \in \mathcal Y)$, which is the solution of ODE system \eqref{eq:ode1}. Next, using the trajectories of the parameter functions and Markov chain $Y_t$, we can achieve the simulations for $\hat{\alpha_t}$ and $\hat{X_t}$. The Matlab code can be found at \url{https://github.com/JiaminJIAN/Regime_switching_MFG}.
	
	\begin{figure}[H]
		\begin{minipage}{\linewidth}
			\centering
			\subcaptionbox{$a_{y}(t)$}
			{\includegraphics[scale=0.36] {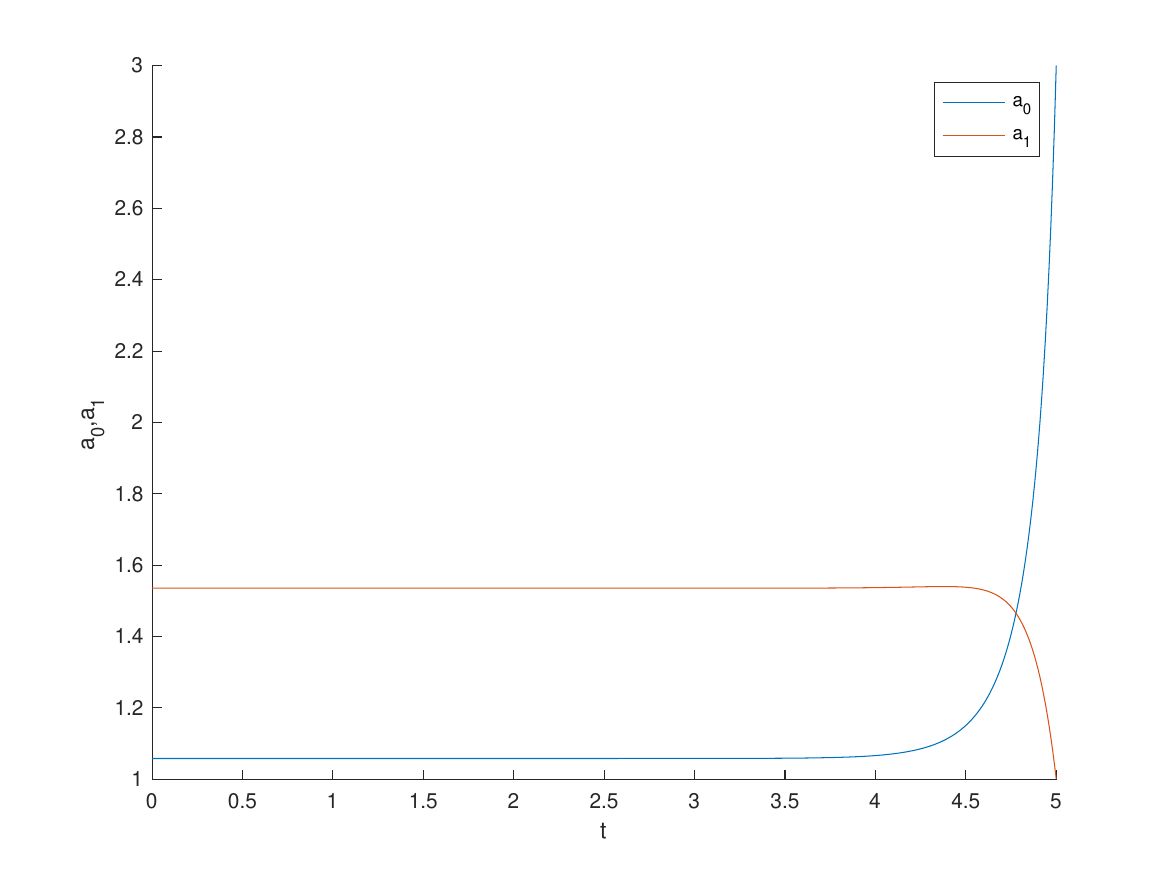}}\quad
			\subcaptionbox{Value function $V$, optimal control $\alpha$, and conditional second moment $\nu$}
			{\includegraphics[scale=0.36] {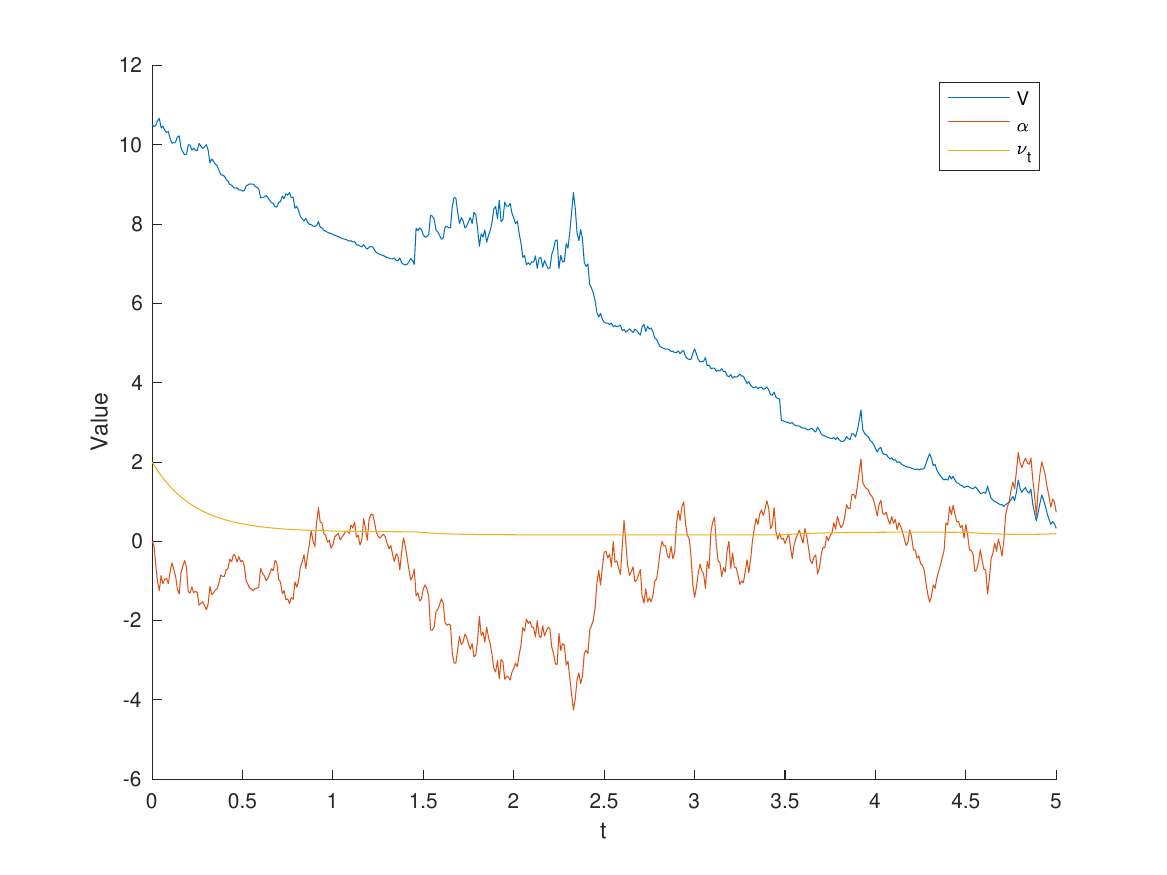}}
			\caption{Simulations for $a_{y}, V, \alpha$ and $\nu$.}
			\label{fig: figure 1}
		\end{minipage}
	\end{figure}
	
	\begin{figure}[H]
		\begin{minipage}{\linewidth}
			\centering
			\subcaptionbox{$b_{y}(t)$}
			{\includegraphics[scale=0.36] {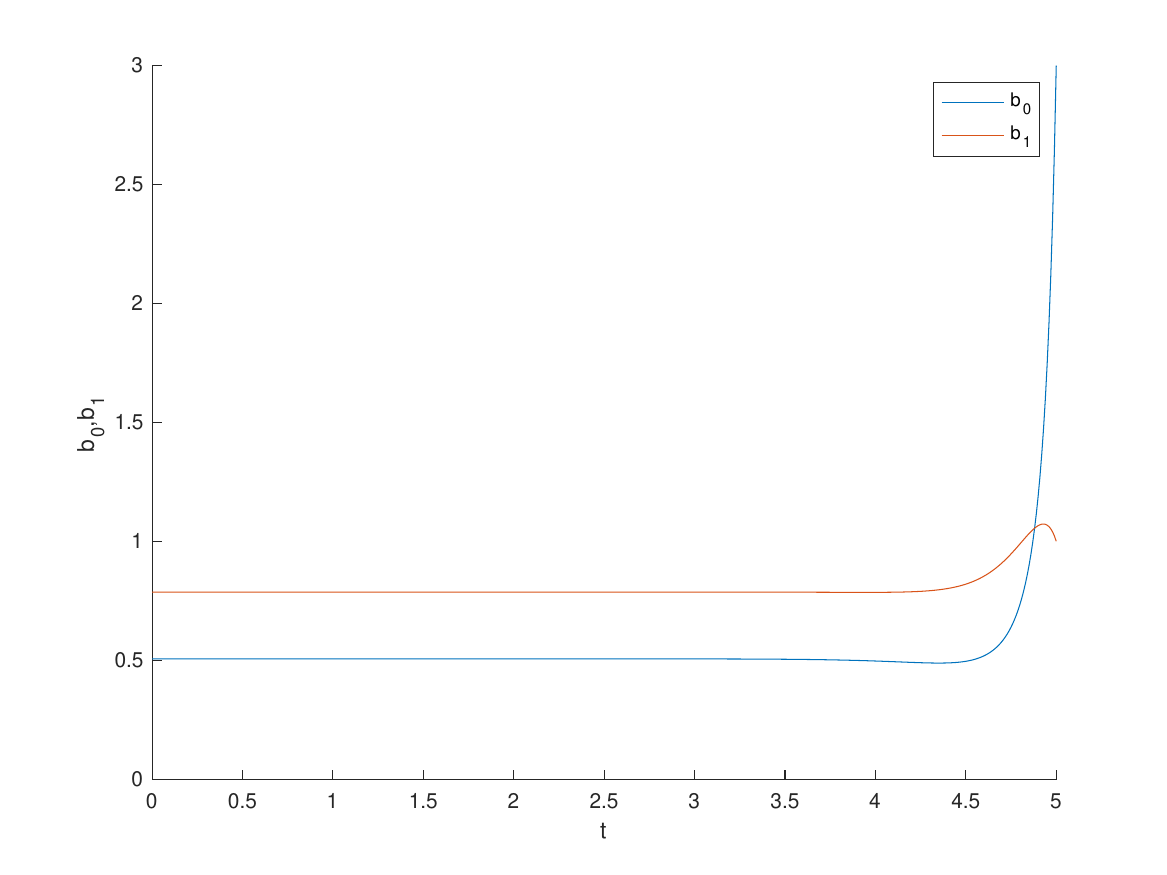}}\quad
			\subcaptionbox{$c_{y}(t)$}
			{\includegraphics[scale=0.36] {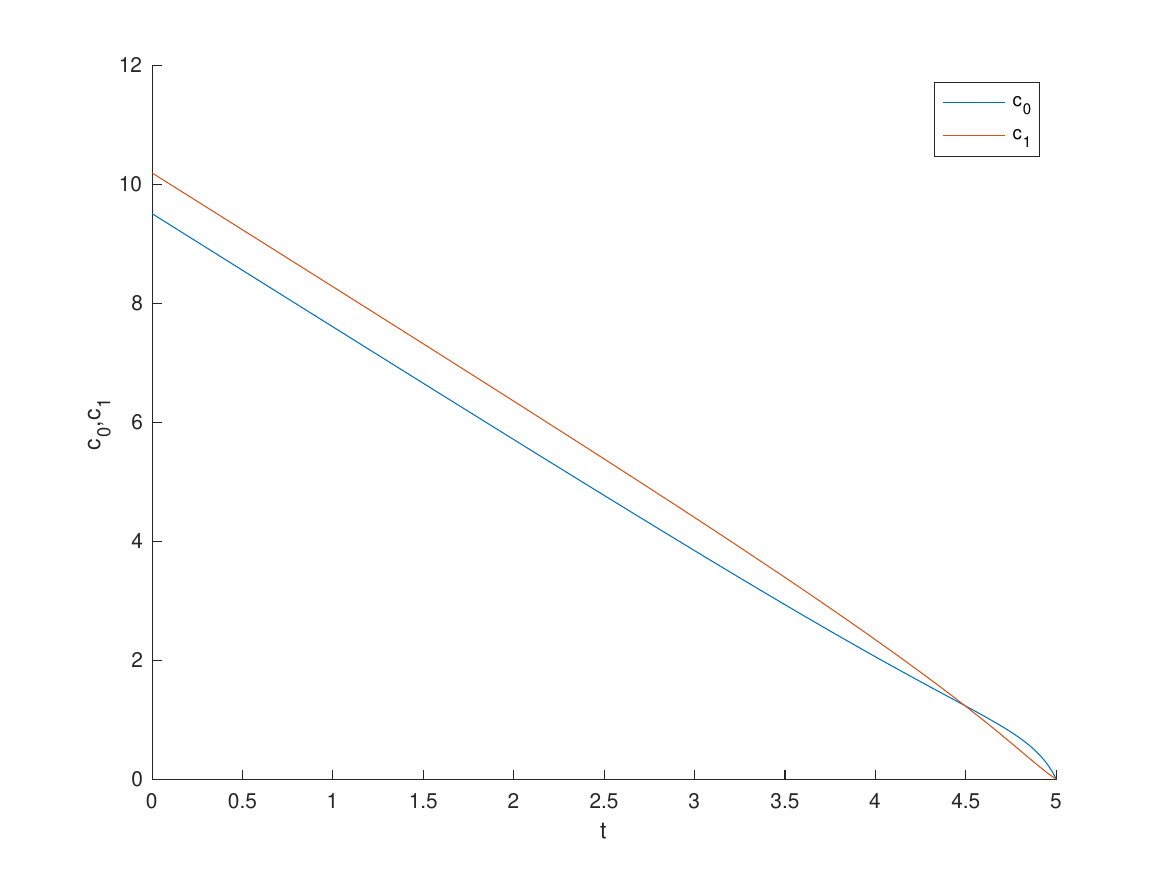}}
			\caption{Simulations for $b_{y}$ and $c_{y}$.}
			\label{fig: figure 2}
		\end{minipage}
	\end{figure}

	As shown in figure \ref{fig: figure 1},  people tend to centralize since the conditional second moment of the population density $\nu_t$ is always decreasing.
	
	\subsection{Convergence of the $N$-player game}
	
	In section \ref{s:section4}, we showed that the generic player's path for the $N$-player game is convergent to the generic player's path for MFGs. In this subsection, we demonstrate the convergence of the conditional first moment, conditional second moment, and the value functions of the $N$-player game to the corresponding terms of the generic player in the Mean Field Game setup by using some numerical examples.
	
	The following figures show the value functions, $\mu^{(N)}$ and $\nu^{(N)}$ under $N \in \{10, 20, 50, 100\}$ with the same parameters' settings as in figure \ref{fig: figure 1} and figure \ref{fig: figure 2} in section \ref{s:simu}. We can clearly see the convergence to the solution of the generic player.
	
	\begin{figure}[H]
		\begin{minipage}{\linewidth}
			\centering
			\subcaptionbox{$\mu_t$: conditional mean of the population density}
			{\includegraphics[scale=0.35] {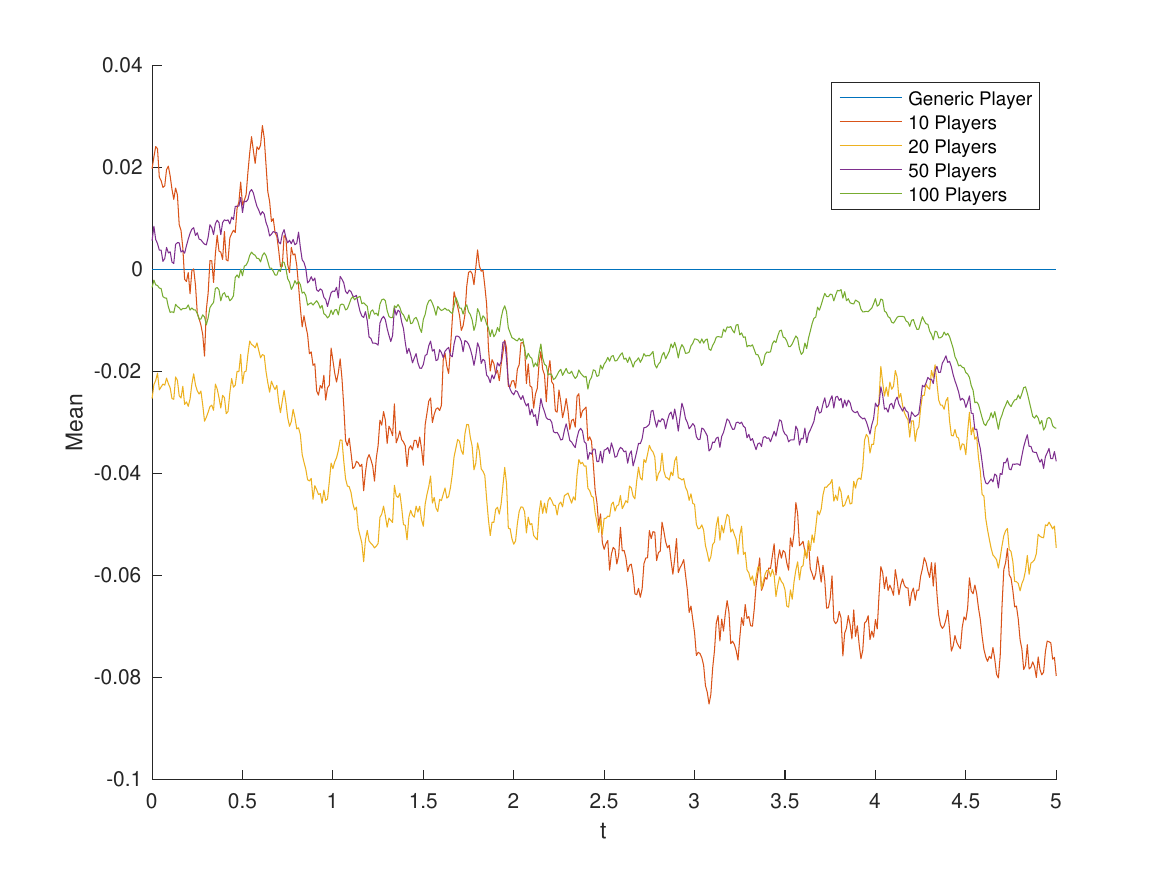}}\quad
			\subcaptionbox{$\nu_t$: conditional 2nd moment of the population density}
			{\includegraphics[scale=0.35] {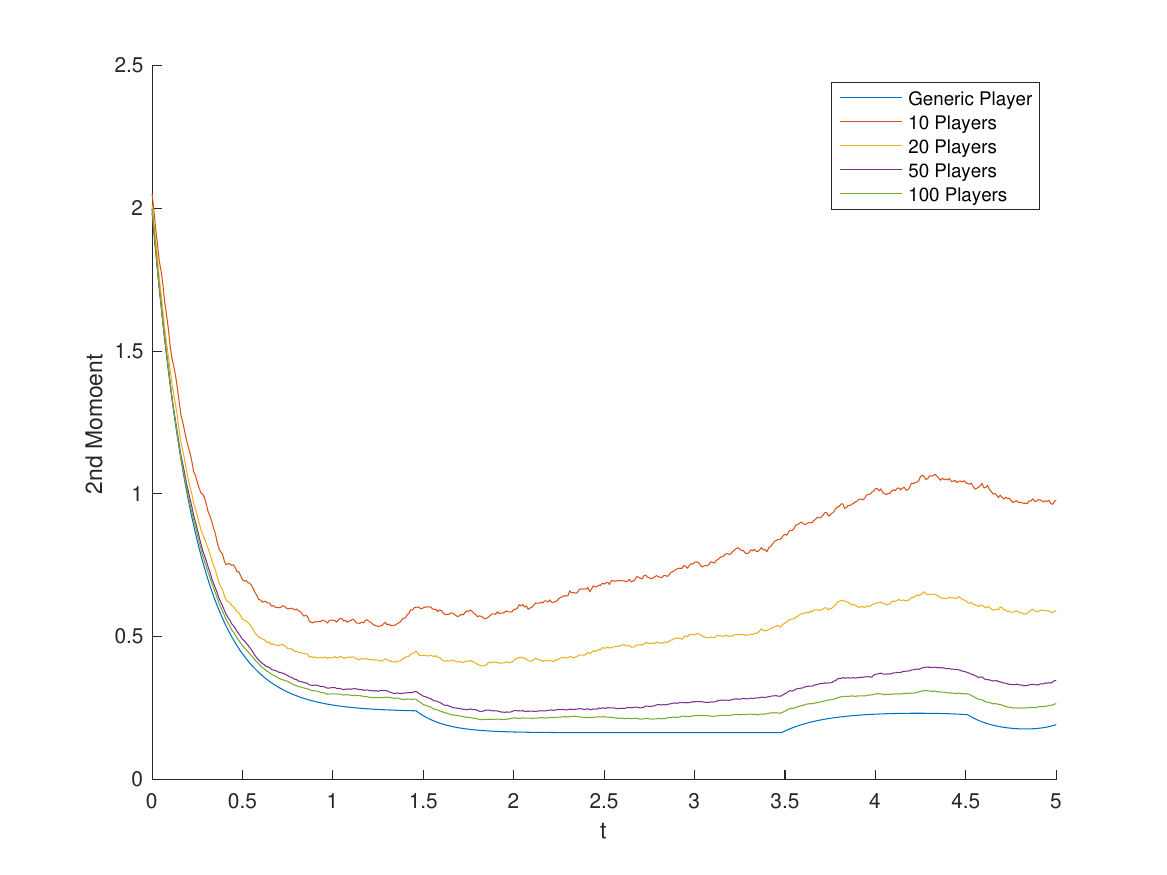}}
			\caption{Simulations for $\mu_t$ and $\nu_t$.}
			\label{fig: figure 3}
		\end{minipage}
	\end{figure}
	\begin{figure}[H]
		\centering
		\includegraphics[scale=0.35] {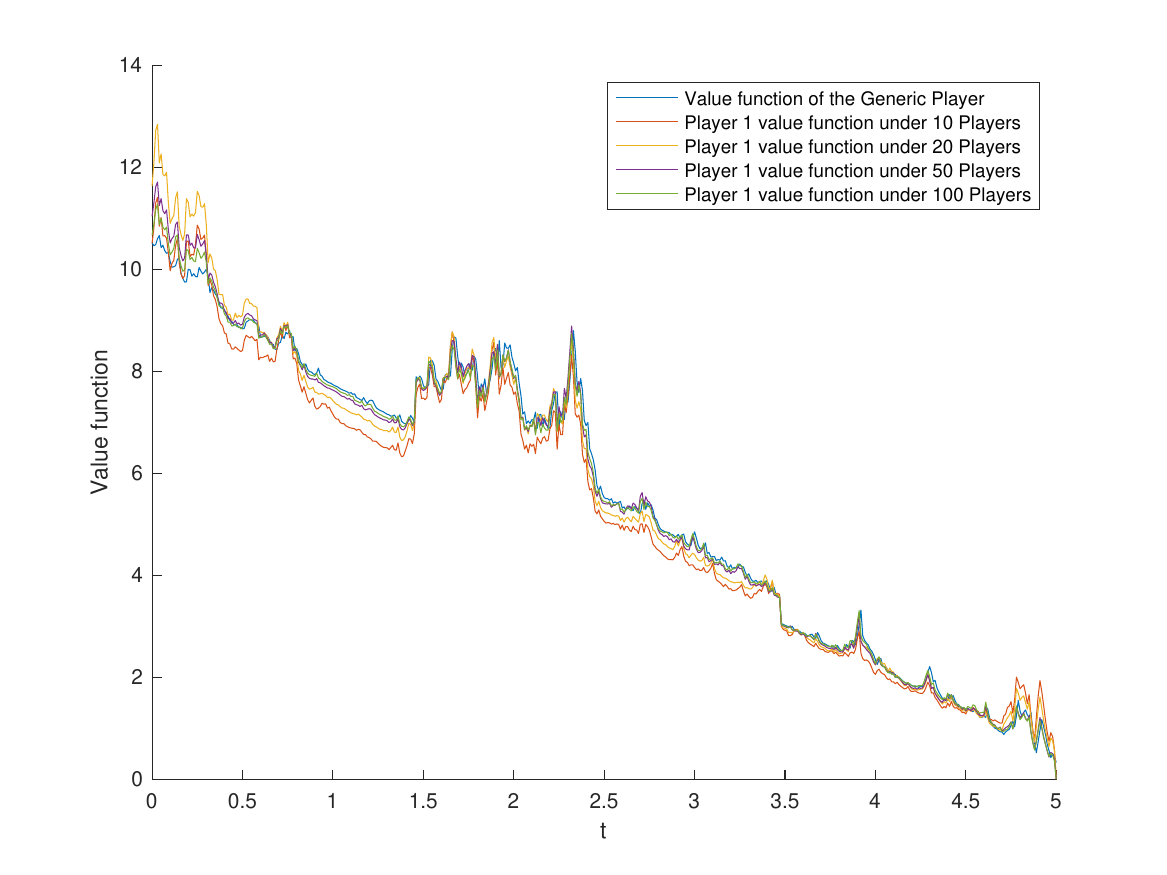}\quad
		\caption{Simulation of player 1's optimal value function $V$.}
		\label{fig: figure 4}
	\end{figure}
	
\section{
Conclusion
}
\label{s:conclusion}

This paper investigates the convergence rate of the $N$-player game, governed by a Markov chain common noise, towards its asymptotic MFG under the Linear-Quadratic-Gaussian structure. To achieve this, firstly, we introduce a Markovian structure using two auxiliary processes for the first and second moments of the MFG equilibrium and employ the fixed point condition in MFG. By doing so, we characterize the equilibrium measure in MFG with a finite-dimensional Riccati system of ODEs. Consequently, we obtain the equilibrium path, equilibrium control, and the value function in MFG.

Subsequently, we address the $N$-player game under the LQG structure, and we characterize its equilibrium path, equilibrium control, and the value function through a Riccati system of ODEs with a dimension of $O(N^3)$. Leveraging the $N$-invariant algebraic structure of this system of ODEs, we establish a dimension reduction result, facilitating a comparison between the equilibrium path $\hat X_1^{(N)}$ in the $N$-player game and the equilibrium path $\hat X$ in the MFG.

To demonstrate the convergence between the two equilibrium paths, we embed $\hat X_1^{(N)}$ from $\Omega^{(N)}$ to $\Omega$ using a distribution copy $Z^N \in \Omega$, leading to the achievement of the convergence result and the computation of the convergence rate. Lastly, some numerical examples are presented to demonstrate the convergence result.

In the future, firstly, we can consider the MFG in more general settings, such as with time delays and Poisson jumps. Next, except for considering the LQG structure, we could consider the convergence of MFG with common noise under more general structures. Furthermore, in this paper, we require positive values for all sensitivities in the cost functional. We find that there is no global solution for MFG when the coefficient of the cost functional is negative, while there is a global solution when the coefficient is positive. So, it is also an interesting problem to investigate the explosion when some sensitivities take negative. 

	\section{Appendix}
	\label{s:appendix}
	
	\subsection{Some explicit solutions on LQG-MFGs}
	\label{s:no_common}
	
	In this part, we only provide explicit solutions to some LQG-MFGs without the common noise.
	The methodology could be the utilization of the standard 
	Stochastic Maximum Principle or Dynamic Programming approach, 
	and all proofs will be omitted. 
	
	Suppose the position of a generic player $X_t$ follows
	\begin{equation*}
		dX_t = \alpha_t dt + \sigma dW_t, \
		X_0 \sim \mathcal{N}(0,1).
	\end{equation*} 
	The goal of the generic player is to minimize the running cost
	\begin{equation*}
		\inf_{\alpha \in \mathcal{A}} \mathbb{E}\left[ \int_0^T\left(\frac{1}{2}\alpha_t^2 + 
		h \int_{\mathbb{R}} (X_t-y)^2m(t, dy)\right)dt \right], 
	\end{equation*}
	subject to
	\begin{equation*}
		m_t = \mathcal{L}aw(X_t), \quad \forall t \in [0, T],
	\end{equation*}
	where $h \in \mathbb R$ is a constant.
	
	Denote 
	$$ 
	V(x,t) = \inf_{\alpha} \mathbb{E}\left[ \left. \int_t^T \left( \frac{1}{2}\alpha_s^2 + h \int_{\mathbb{R}} 
	(X_s-y)^2m(s, dy) \right) ds \right \rvert   X_t = x\right].$$ 
	Note that the model can be characterized by Hamilton-Jacobian-Bellman equation coupled by 
	Fokker-Planck-Kolmogorov equation:
	$$\begin{array}{ll}
		\begin{cases}
			\partial_{t} V + \frac{1}{2}\sigma^2 \partial_{xx}V - \frac{1}{2} (\partial_{x} V)^2 + F(x,m) = 0, & 
			(t,x) \in [0,T]\times \mathbb{R}, \\
			\partial_{t} m - \frac{1}{2} \sigma^2 \partial_{xx}m - \partial_x(m \partial_{x} V) = 0, & 
			(t,x) \in [0,T]\times \mathbb{R}, \\
			m_0 \sim \mathcal{N}(0,1), V(x,T) = 0, & x\in \mathbb{R},
		\end{cases}
	\end{array}
	$$
	where $F(x,m) = h \int_{\mathbb{R}} (x-y)^2 m(dy)$.
	
	The monotonicity condition on the source term $F$ in the variable $m$ plays a crucial role in the uniqueness of 
	the MFG system. A monotone function $f:\mathbb R\mapsto \mathbb R$ is
	said to be increasing if it satisfies $(f(x_1) - f(x_2))(x_1 - x_2) \ge 0$ , and decreasing if $-f$ is increasing.
	This definition can be generalized to an infinite dimensional function $F(x, m)$.
	\begin{definition}
		The real function $F$ on $\mathbb{R} \times\mathcal{P}_2(\mathbb{R})$ is said to be monotone, if, for all $m \in \mathcal{P}_2(\mathbb{R})$, the mapping $\mathbb{R} \ni x \mapsto F(x,m) $ is at most of quadratic growth, 
		and for all $m_1$, $m_2$ it satisfies
		$$
		\int_{\mathbb{R}}\left( F(x,m_1)-F(x,m_2) \right)d(m_1-m_2)(x) \ge 0. 
		$$
		$F$ is said to be anti-monotone, if $(-F)$ is monotone.
	\end{definition}
	According to \cite{Car10}, if $F$ is monotone, then MFGs have at most one solution.
	Interestingly, the monotonicity of $F$ is dependent on the sign of $h$.
	\begin{lemma}
		\label{l:F}
		$F(x,m) = h \int_{\mathbb R} (x-y)^2 m(dy)$ is monotone if $h<0$, and anti-monotone if $h>0$.
	\end{lemma}
	A natural question is how the MFG system behaves differently to the monotonicity of $F$?
	
	\subsubsection{Case I: $h > 0$}
	\begin{lemma}\label{l:h>0}
		For $h > 0$, there exists a solution (may not be unique) to the MFG system in the form of $V(x,t) = f_1(t)x^2 + f_3(t)$ and $m(t) \sim \mathcal{N}(0,\gamma(t))$, where
		\begin{align*}
			& f_1(t) = \sqrt{\frac{h}{2}}\frac{1-e^{-2\sqrt{2h}(T-t)}}{1+e^{-2\sqrt{2h}(T-t)}} \text{ , } \gamma(t) =  e^{-\int_0^t 4f_1(s)ds} \left(1+ \int_0^t \sigma^2 e^{\int_0^s 4f_1(u)du} ds \right),\\
			& f_3(t) = \int_t^T (\sigma^2 f_1(s) + h \gamma(s)) ds.
		\end{align*}
	\end{lemma}
	
	\subsubsection{Case II: $h<0$}
	\begin{lemma}\label{l:h<0}
		For $h < 0$, there exists a unique solution  in $( t_0,T ]$ to the MFG system in the form of 
		$V(x,t) = g_1(t)x^2 + g_3(t)$ and $m(t) \sim \mathcal{N}(0,\lambda(t))$, where
		\begin{align*}
			& g_1(t) = - \sqrt{- \frac{h}{2}}\tan\left(\sqrt{-2h}(T-t)  \right) \text{ , } \lambda(t) = e^{-\int_0^t 4g_1(s)ds} \left(1+ \int_0^t \sigma^2 e^{\int_0^s 4g_1(u)du} ds \right), \\
			& g_3(t) = \int_t^T (\sigma^2 g_1(s) + h \lambda(s)) ds \text{ , } t_0 = max\left(0,T-\frac{1}{\sqrt{-2h}}\frac{\pi}{2}\right).
		\end{align*}
	\end{lemma}
	
	\subsubsection{Remark}
	When $h>0$, the cost is anti-monotone, and there exists at least one global solution.
	When $h<0$, the cost is monotone, and there exists at most one solution. 
	Unfortunately, this solution lives in a short period. 
	Lemma \ref{l:h<0} coincides with the notes in Section 3.8 of \cite{carmona2018} saying that due to the opposite time evolution of the system of HJB-FPK, the existence of the solution may exist for only a short period.

	\subsection{Dynkin's formula for a regime-switching diffusion with a quadratic function} 
	\label{s:dynkin}
	
	Since the running cost \eqref{eq:running cost} has a quadratic growth in the state variable, the value function $V[\hat m](y, x, t)$ is expected to possess similar growth.
	Next, we present a version of Dynkin's formula for the functions of quadratic growth, which is sufficient for our purpose. Throughout this subsection, we will use $K$ in various places as a generic constant that varies from line to line.
	The notions of this subsection are independent of other parts of the paper.
	
	\begin{lemma}
		\label{l:chain rule}
		Let $X$ be the $\mathbb R^d$-valued process satisfying
		$$ X_t = X_0 + \int_0^t \left( \tilde{b}_1(Y_s, s) X_s + \tilde{b}_2(Y_s, s) \alpha_s \right) ds + \int_0^t \sigma(s) dW_s,$$
		where $Y$ is CTMC with a generator 
		$$Y \sim Q = (q_{ij})_{i, j = 1, 2, \ldots, \kappa},$$
		Suppose $\sigma(\cdot)$, $\tilde{b}_1(y, \cdot)$ and $\tilde{b}_2(y, \cdot)$ are continuous functions on $[0, T]$ 
		for every $y \in \mathcal Y :=\{1,2, \ldots, \kappa\}$. 
		If $X_0\in L^4$, $\alpha\in L_\mathbb F^4$ and 
		$f:\mathcal Y \times  R^d \times \mathbb R \mapsto \mathbb R$ satisfies, for some large $K$
		$$\sup_{y\in \mathcal Y, t\in [0, T]} 
		\{|f(y, x, t)| + (1 + |x|) |\nabla f(y, x, t)| + (1 + |x|)^2 |\Delta f(y, x, t)| +  |\partial_t f(y, x, t)| \}
		\le K(|x|^2 +1), $$
		then the following identity holds for all $t\in [0, T]$:
		\begin{equation*}
			\mathbb E \left[f(Y_t, X_t, t)\right] = 
			\mathbb E \left[f(Y_0, X_0, 0)\right] +
			\mathbb E \left[ \int_0^t (\partial_t + \mathcal L^{\alpha_s} + \mathcal Q) f(Y_s, X_s, s) ds \right],
		\end{equation*}
		where
		$$\mathcal L^a f(y, x, s) = \left(\frac 1 2 \text{Tr} \left(\sigma_s \sigma_s^\top \Delta \right)  + \left(\tilde{b}_{1y} x + \tilde{b}_{2y} a \right) \cdot \nabla_x \right) f(y, x, s)$$
		and
		$$\mathcal Q f(y, x, s) = \sum_{i=1}^n q_{y, i} f(i, x, s).$$
		
	\end{lemma}
	\begin{proof}
		It's enough to show that the local martingale defined by It\^{o}'s formula
		\begin{equation}
			\label{eq:Mtf}
			M_t^f = f(Y_t, X_t, t) - f(Y_0, X_0, 0) - \int_0^t (\partial_t + \mathcal L^{\alpha_s} + \mathcal Q) f(Y_s, X_s, s) ds
		\end{equation}
		is uniformly integrable, hence is a true martingale. 
		
		First, note that from the assumptions on $X_0$ and $\alpha$, we have
		\begin{equation*}
		\begin{aligned}
	    \mathbb E \left[ \|X_t\|^4 \right]
		& \leq K \mathbb E \left[ \| X_0 \|^4 + \int_0^t \|\tilde{b}_1(Y_s, s) X_s + \tilde{b}_2(Y_s, s) \alpha_s\|^4 ds + \int_0^t \|\sigma_s W_s\|^4 ds \right] \\
		& \leq K \mathbb E \left[ \| X_0 \|^4 + \int_0^t \| X_s \|^4 ds + \int_0^t \| \alpha_s \|^4 ds + \int_0^t \|\sigma_s W_s\|^4 ds \right] \\
		& \leq K + K \int_0^t \mathbb E  \left[ \|X_s\|^4 \right] ds,
		\end{aligned}
		\end{equation*}
		where $K$ is a generic constant that varies from line to line. Then, by the Gr\"onwall's inequality, 
		$$\mathbb E \left[ \|X_t\|^4 \right] \leq Ke^{Kt} \leq K,$$
		which implies that $\{X_t: 0\le t\le T\}$ is $L^4$ bounded uniformly in $t$.
		
		On the other hand, since $x \mapsto f(y, x, t)$ is at most 
		quadratic growth uniformly in $(y, t)$, 
		we conclude that $f (Y_t, X_t, t)$ is uniformly  $L^2$ bounded  from the fact
		$$\sup_{t \in [0,T]} \mathbb E \left[f^2 (Y_t, X_t, t) \right] \le K \sup_{t\in [0, T]} \mathbb E \left[ \|X_t\|^4 \right] + K \le K.$$
		The uniform $L^2$-boundedness of $\int_0^t \partial_t f(Y_s, X_s, s) ds$ follows from our assumption on $\partial_t f$.
		Similarly, since $\mathcal Q f$ has a quadratic growth uniformly in $y$ 
		and $t$, and $$\left\{\int_0^t \mathcal Q f(Y_s, X_s, s) ds: 0\le t\le T\right\}$$ is $L^2$ bounded.
		At last, we have
		$$\displaystyle \begin{array}
			{ll}
			\displaystyle \mathbb E \left[ \left(\int_0^t \mathcal{L}^{\alpha_s} f(Y_s, X_s, s) ds \right)^2 \right]  
			\\ \displaystyle  \le 
			K \mathbb E \left[ \int_0^t \left(\left(\tilde{b}_1(Y_s, s) X_s + \tilde{b}_2(Y_s, s) \alpha_s \right)\cdot \nabla f + \frac 1 2 \text{Tr} \left(\sigma_s\sigma_s^\top \Delta f \right) \right)^2(Y_s, X_s, s) ds \right] 
			\\
			\displaystyle  \le K \mathbb E \left[ \int_0^t \|\tilde{b}_1(Y_s, s) X_s + \tilde{b}_2(Y_s, s) \alpha_s\|^2 \|\nabla f\|^2 (Y_s, X_s, s) ds \right] 
			\\ \displaystyle \hspace{1in}
			+ K \mathbb E \left[ \int_0^t  \frac 1 4 \|\text{Tr} \left(\sigma_s\sigma_s^\top \Delta f \right)\|^2(Y_s, X_s, s) ds \right] \\
			\displaystyle  \le K \mathbb E \left[ \int_0^t  \|\alpha_s\|^4  ds \right] + K \mathbb E \left[ \int_0^t  \|X_s\|^4  ds \right] 
			+K \mathbb E \left[ \int_0^t  |\nabla f|^4 (Y_s, X_s, s) ds \right] 
			\\ \displaystyle \hspace{1in} + K \mathbb E \left[ \int_0^t  \frac 1 4 \|\text{Tr}\Delta f\|^2 \left(Y_s, X_s, s \right) ds \right]. \\
		\end{array}
		$$
		Since $\nabla f$ is linear growth in $x$, the second term $\sup_{t \in [0, T]} \mathbb E \left[\int_0^t  \|\nabla f\|^4 (Y_s, X_s, s) ds\right]$ is finite.
		Together with assumptions on $\Delta f$ and $\alpha$, we have uniform $L^2$-boundedness of $\int_0^t \mathcal L^{\alpha_s} f (Y_s, X_s, s) ds$.
		
		As a result, each term of the right-hand side of \eqref{eq:Mtf} is uniform $L^2$-bounded in $t$, and thus $M_t^f$ belongs to $L^2_{\mathbb F}$ and this implies the uniform integrability.
	\end{proof}
	
	\subsection{Proof of the property of G}
	\label{s:appendix lipschitz}
	
	\begin{lemma}
	    Define
		$$\mathcal E_t(\phi)
		= \exp \left\{ \int_0^t \phi_s ds \right\},
		$$
		and
	\begin{equation*}
	    G_t(x, \phi_1, \phi_2, \phi_3, W) = x\mathcal E_t(\phi_1 - \phi_2) + \mathcal E_t(\phi_1 - \phi_2) \int_0^t \mathcal E_s (-\phi_1 + \phi_2) \left(\phi_2(s) \phi_3(s) ds + dW_s \right),
	\end{equation*}
	where $x$ is a given constant, $\phi_1, \phi_2, \phi_3$ are RCLL functions on $[0, T]$.
	Then
	\begin{equation*}
	    \begin{aligned}
	       & \mathbb E \left[ \left\vert G_t(x^1, \phi_1, \phi_2^1, \phi_3^1, W) - G_t(x^2, \phi_1, \phi_2^2, \phi_3^2, W) \right\vert^2 \right] \\
	       \leq & K \left(|x^1 - x^2|^2 + \sup_{0 \leq t \leq T}  |\phi_2^1(t) - \phi_2^2(t)|^2 + \sup_{0 \leq t \leq T} 
	       | \phi_3^1(t) - \phi_3^2(t) |^2 \right).
	    \end{aligned}
	    \end{equation*}
	\end{lemma}
	\begin{proof}
	    Firstly, it can be shown that $G(\cdot, \phi_1, \phi_2, \phi_3, W)$ is Lipschitz continuous with respect to $x$
	    \begin{equation*}
	    \begin{aligned}
	        \mathbb E \left[ \left\vert G_t(x^1, \phi_1, \phi_2, \phi_3, W) - G(x^2, \phi_1, \phi_2, \phi_3, W) \right\vert \right] & \leq  \left\vert x^1 \mathcal{E}_t(\phi_1 - \phi_2) - x^2 \mathcal{E}_t(\phi_1 - \phi_2) \right\vert \\
	        & \leq \mathcal{E}_t(\phi_1 - \phi_2) |x^1 - x^2| \\
	        & \leq K(|\phi_1|_{\infty}, |\phi_2|_{\infty}, T) \mathbb |x^1 - x^2|.
	    \end{aligned}
	    \end{equation*}
	    Next, we have
	    \begin{equation*}
	    \begin{aligned}
	        & \mathbb E \left[ \left\vert G_t(x, \phi_1, \phi_2, \phi_3^1, W) - G(x, \phi_1, \phi_2, \phi_3^2, W) \right\vert^2 \right] \\
	        = &  \left\vert  \mathcal{E}_t(\phi_1 - \phi_2) \int_0^t \mathcal{E}_s(\phi_1 - \phi_2) \phi_2(s) (\phi_3^1(s) - \phi_3^2(s)) ds \right\vert^2 \\
	        \leq & \mathcal{E}_t(2\phi_1 - 2\phi_2) \left( \int_0^t \mathcal{E}_s(\phi_1 - \phi_2) |\phi_2(s)| |(\phi_3^1(s) - \phi_3^2(s))| ds \right)^2 \\
	        \leq &  K(|\phi_1|_{\infty}, |\phi_2|_{\infty}, T)  \left( \int_0^T |\phi_3^1(s) - \phi_3^2(s)| ds \right)^2 \\
	        \leq & K(|\phi_1|_{\infty}, |\phi_2|_{\infty}, T) 
	         \sup_{0 \leq t \leq T} \left\vert \phi_3^1(t) -\phi_3^2(t) \right \vert^2.
	    \end{aligned}
	    \end{equation*}
	   Similarly, for $\phi_2^1(\cdot), \phi_2^2(\cdot) \in C([0, T])$, 
	    \begin{equation*}
	    \begin{aligned}
	        & \mathbb E \left[ \left\vert G_t(x, \phi_1, \phi_2^1, \phi_3, W) - G(x, \phi_1, \phi_2^2, \phi_3, W) \right\vert^2 \right] \\
	        \leq & K \left\vert  x \mathcal E_t(\phi_1 - \phi_2^1) - x \mathcal E_t(\phi_1 - \phi_2^2)  \right\vert^2 \\
	        & \hspace{0.2 in} + K  \left\vert \mathcal E_t(\phi_1 - \phi_2^1) \int_0^t \mathcal E_s(-\phi_1 + \phi_2^1) \phi_2^1(s) \phi_3(s) ds - \mathcal E_t(\phi_1 - \phi_2^2) \int_0^t \mathcal E_s(-\phi_1 + \phi_2^2) \phi_2^2(s) \phi_3(s) ds  \right\vert^2 \\
	        & \hspace{0.2 in} + K \mathbb E \left[ \left\vert \mathcal E_t(\phi_1 - \phi_2^1) \int_0^t \mathcal E_s(-\phi_1 + \phi_2^1) d W_s - \mathcal E_t(\phi_1 - \phi_2^2) \int_0^t \mathcal E_s(-\phi_1 + \phi_2^2) d W_s  \right\vert^2  \right] \\
	        := & K(J_1 + J_2 + J_3).
	    \end{aligned}
	    \end{equation*}
	    Note that by the mean-value theorem and the continuity of $\phi_1, \phi_2^1$ and $\phi_2^2$ on $[0, T]$, we can get
	    \begin{equation*}
	    \begin{aligned}
	       J_1 &= \left\vert  x \mathcal E_t(\phi_1 - \phi_2^1) - x \mathcal E_t(\phi_1 - \phi_2^2)  \right\vert^2 \\
	        &= x^2 \left(e^{\int_0^t (\phi_1(s) - \phi_2^1(s))ds} - e^{\int_0^t (\phi_1(s) - \phi_2^2(s))ds} \right)^2 \\
	        & \leq K\left(x, \left\vert \phi_2^1 \right\vert_{\infty}, \left\vert \phi_2^2 \right\vert_{\infty}, T  \right) e^{\int_0^t 2\phi_1(s) ds} \left\vert \phi_2^1 - \phi_2^2 \right\vert_{\infty}^2 \\
	        & \leq K \left(x, |\phi_1|_{\infty}, \left\vert \phi_2^1 \right\vert_{\infty}, \left\vert \phi_2^2 \right\vert_{\infty}, T  \right) \left\vert \phi_2^1 - \phi_2^2 \right\vert_{\infty}^2,
	    \end{aligned}
	    \end{equation*}
	    and 
	    \begin{equation*}
	    \begin{aligned}
	       J_3 &= \mathbb E \left[ \left\vert \mathcal E_t(\phi_1 - \phi_2^1) \int_0^t \mathcal E_s(-\phi_1 + \phi_2^1) d W_s - \mathcal E_t(\phi_1 - \phi_2^2) \int_0^t \mathcal E_s(-\phi_1 + \phi_2^2) d W_s  \right\vert^2 \right] \\
	       &= \mathbb E \left[ \left\vert \mathcal E_t(\phi_1 - \phi_2^1) \int_0^t \mathcal E_s(-\phi_1 + \phi_2^1) d W_s - \mathcal E_t(\phi_1 - \phi_2^1) \int_0^t \mathcal E_s(-\phi_1 + \phi_2^2) d W_s  \right. \right.  \\
	       & \hspace{0.5in} +  \left. \left. \mathcal E_t(\phi_1 - \phi_2^1) \int_0^t \mathcal E_s(-\phi_1 + \phi_2^2) d W_s - \mathcal E_t(\phi_1 - \phi_2^2) \int_0^t \mathcal E_s(-\phi_1 + \phi_2^2) d W_s  \right\vert^2 \right] \\
	       &\leq 2 \mathcal E_t(2\phi_1 - 2\phi_2^1) \int_0^t \left(\mathcal E_s(-\phi_1 + \phi_2^1)- \mathcal E_s(-\phi_1 + \phi_2^2) \right)^2 ds \\
	       & \hspace{0.5in} + 2 \left(\mathcal E_t(\phi_1 - \phi_2^1) - \mathcal E_t(\phi_1 - \phi_2^2) \right)^2 \int_0^t \mathcal E_s(-2\phi_1 +2 \phi_2^2) ds \\
	       & \leq K \left(|\phi_1|_{\infty}, \left\vert \phi_2^1 \right\vert_{\infty}, \left\vert \phi_2^2 \right\vert_{\infty}, T  \right) \left\vert \phi_2^1 - \phi_2^2 \right\vert_{\infty}^2.
	    \end{aligned}
	    \end{equation*}
	    Lastly, using the similar argument, we have
	    \begin{equation*}
	    \begin{aligned}
	       J_2 &= \left\vert \mathcal E_t(\phi_1 - \phi_2^1) \int_0^t \mathcal E_s(-\phi_1 + \phi_2^1) \phi_2^1(s) \phi_3(s) ds - \mathcal E_t(\phi_1 - \phi_2^2) \int_0^t \mathcal E_s(-\phi_1 + \phi_2^2) \phi_2^2(s) \phi_3(s) ds  \right\vert^2  \\
	       &= \left\vert \mathcal E_t(\phi_1 - \phi_2^1) \int_0^t \mathcal E_s(-\phi_1 + \phi_2^1) \phi_2^1(s) \phi_3(s) ds - \mathcal E_t(\phi_1 - \phi_2^2) \int_0^t \mathcal E_s(-\phi_1 + \phi_2^1) \phi_2^1(s) \phi_3(s) ds  \right. \\
	       & \hspace{0.5in} +  \left.  \mathcal E_t(\phi_1 - \phi_2^2) \int_0^t \mathcal E_s(-\phi_1 + \phi_2^1) \phi_2^1(s) \phi_3(s) ds - \mathcal E_t(\phi_1 - \phi_2^2) \int_0^t \mathcal E_s(-\phi_1 + \phi_2^2) \phi_2^2(s) \phi_3(s) ds  \right\vert^2 \\
	       &\leq 2 \left\vert \left( \mathcal E_t(\phi_1 - \phi_2^1) - \mathcal E_t(\phi_1 - \phi_2^2) \right) \int_0^t \mathcal E_s(-\phi_1 + \phi_2^1) \phi_2^1(s) \phi_3(s) ds \right\vert^2  \\
	       & \hspace{0.5in} + 2  \left\vert \mathcal E_t(\phi_1 - \phi_2^2) \left( \int_0^t \mathcal E_s(-\phi_1 + \phi_2^1) \phi_2^1(s) \phi_3(s) ds - \int_0^t \mathcal E_s(-\phi_1 + \phi_2^2) \phi_2^2(s) \phi_3(s) ds  \right) \right\vert^2 \\
	       & \leq K \left(|\phi_1|_{\infty}, \left\vert \phi_2^1 \right\vert_{\infty}, \left\vert \phi_2^2 \right\vert_{\infty}, \left\vert \phi_3 \right\vert_{\infty}, T \right) \left\vert \phi_2^1 - \phi_2^2 \right\vert_{\infty}^2 \\
	       & \hspace{0.5in} + 2  \left\vert \mathcal E_t(\phi_1 - \phi_2^2) \left( \int_0^t \mathcal E_s(-\phi_1 + \phi_2^1) \phi_2^1(s) \phi_3(s) ds - \int_0^t \mathcal E_s(-\phi_1 + \phi_2^2) \phi_2^1(s) \phi_3(s) ds  \right) \right. \\
	       & \hspace{0.5in} + \left. \mathcal E_t(\phi_1 - \phi_2^2) \left( \int_0^t \mathcal E_s(-\phi_1 + \phi_2^2) \phi_2^1(s) \phi_3(s) ds - \int_0^t \mathcal E_s(-\phi_1 + \phi_2^2) \phi_2^2(s) \phi_3(s) ds  \right) \right\vert^2 \\
	       & \leq K \left(|\phi_1|_{\infty}, \left\vert \phi_2^1 \right\vert_{\infty}, \left\vert \phi_2^2 \right\vert_{\infty}, \left\vert \phi_3 \right\vert_{\infty}, T \right) \left\vert \phi_2^1 - \phi_2^2 \right\vert_{\infty}^2.
	    \end{aligned}
	    \end{equation*}
	    Sum up the above inequalities for $J_1, J_2$ and $J_3$, then
	    $$\mathbb E \left[ \left\vert G_t(x, \phi_1, \phi_2^1, \phi_3, W) - G(x, \phi_1, \phi_2^2, \phi_3, W) \right\vert^2 \right] \leq K \left(x, |\phi_1|_{\infty}, \left\vert \phi_2^1 \right\vert_{\infty}, \left\vert \phi_2^2 \right\vert_{\infty}, \left\vert \phi_3 \right\vert_{\infty}, T \right) \left\vert \phi_2^1 - \phi_2^2 \right\vert_{\infty}^2.$$
	    Thus, we can obtain the desired result.

	\end{proof}

	\subsection{Proof of the existence and uniqueness of the ODE system}
	
	Consider the following ODE system
	\begin{equation}
		\label{eq:a}
		\begin{aligned}
			\begin{cases}
				\displaystyle a_y' + C_1 \tilde{b}_{1y} a_y - C_2 \tilde{b}_{2y}^2 a_y^2 + \sum_{i=1}^{\kappa} q_{y,i} a_i + h_y = 0, \\
				\displaystyle a_y(T) = g_y,
			\end{cases}
		\end{aligned}
	\end{equation}
	for $y \in \mathcal Y = \{1, 2, \dots, \kappa\}$, where $C_1, C_2, h_y, g_y$ are in $\mathbb R^+$. We need to show the existence and uniqueness of the solution to \eqref{eq:a}. Define $T_y^{(N)}$ as
	\begin{align*}
		& T_y^{(N)}[a](t)  = \left[\left(g_y + \int_t^T \left(h_y + C_1 \tilde{b}_{1y}(s) a_y(s) - C_2 \tilde{b}^2_{2y}(s) a_y^2(s) + \sum_{i=1}^{\kappa}q_{y,i}a_i(s) \right) ds \right) \wedge N\right]\vee 0,
	\end{align*}
	where $a = [a_1, a_2, \dots, a_{\kappa}]^\top$. Let $D = \{ f \in C([0,T]): 0 \le \sup_{t \in [0,T]}f(t)\le N \}$. Note that $T_{y}^{(N)}(y \in \mathcal Y)$ maps $D^\kappa$ to $D^\kappa$.
	\begin{lemma}
		\label{l:existN}
		For fixed $N$, there exists a unique solution in $C([0,T])$ to 
		\begin{equation}
			\label{aN}
			a = T_y^{(N)}[a].
		\end{equation}
	\end{lemma}
	\begin{proof}
		Denote the norm $\|f \|_k = \left\| e^{kt} \max_{y \in \mathcal Y}\lvert f_{y}\rvert\right\|_\infty$, where $k$ needs to be determined later and $f$ is a $\kappa$ dimensional vector with entry of $f_{y}, y \in \mathcal Y$, which is equivalent to the infinite norm. Define the iteration rule $a_{y}^{(n+1)} = T_{y}^{(N)}\left[a_y^{(n)}\right]$ for $y \in \mathcal Y$. Note that
		\begin{equation*}
			\begin{aligned}
				& \left\| e^{kt}\left(a_y^{(n+1)}(t) - a_y^{(n)}(t) \right) \right\|_\infty \\
				\le & \sup_{t \in[0,T]} e^{kt}\int_t^T \left(C_1 \left \lvert \tilde{b}_{1y} \right\rvert_{\infty} \left \lvert a_y^{(n)}(s) - a_y^{(n-1)}(s) \right\rvert + C_2 \left \lvert \tilde{b}_{2y} \right\rvert^2_{\infty} \left\lvert \left(a_y^{(n)}(s)\right)^2-\left(a_y^{(n-1)}(s)\right)^2\right\rvert \right. \\
				& \hspace{1 in} + \left. \sum_{i=1}^{\kappa} q_{y,i} \left\lvert a_i^{(n)}(s)-a_i^{(n-1)}(s)\right\rvert \right)  ds \\
				\le & \sup_{t \in[0,T]} e^{kt}\int_t^T \left(C_1 \left \lvert \tilde{b}_{1y} \right\rvert_{\infty} \left \lvert a_y^{(n)}(s) - a_y^{(n-1)}(s) \right\rvert + 2NC_2 \left \lvert \tilde{b}_{2y} \right\rvert^2_{\infty} \left\lvert a_y^{(n)}(s)-a_y^{(n-1)}(s)\right\rvert \right. \\
				& \hspace{1 in} + \left. \sum_{i=1}^{\kappa} q_{y,i} \left\lvert a_i^{(n)}(s)-a_i^{(n-1)}(s)\right\rvert \right)  ds \\
				\le & \sup_{t \in[0,T]} e^{kt}\int_t^T e^{-ks}\left( C_1 \left \lvert \tilde{b}_{1y} \right\rvert_{\infty} +  2NC_2 \left \lvert \tilde{b}_{2y} \right\rvert^2_{\infty} + \kappa \max_{i \in \mathcal Y} |q_{y,i}| \right) \left\| a^{(n)}- a^{(n-1)} \right\|_k ds \\
				\le & \frac{C_1 \left \lvert \tilde{b}_{1y} \right\rvert_{\infty} +  2NC_2 \left \lvert \tilde{b}_{2y} \right\rvert^2_{\infty} + \kappa \max_{i \in \mathcal Y} |q_{y,i}|}{k} \left\| a^{(n)}- a^{(n-1)} \right\|_k.
			\end{aligned}
		\end{equation*}
		Choose $k > C_1 \left \lvert \tilde{b}_{1y} \right\rvert_{\infty} +  2NC_2 \left \lvert \tilde{b}_{2y} \right\rvert^2_{\infty} + \kappa \max_{i \in \mathcal Y} |q_{y,i}|$, then 
		\begin{align*}
			\left\| a^{(n+1)}- a^{(n)} \right\|_k \le \frac{C_1 \left \lvert \tilde{b}_{1y} \right\rvert_{\infty} +  2NC_2 \left \lvert \tilde{b}_{2y} \right\rvert^2_{\infty} + \kappa \max_{i \in \mathcal Y} |q_{y,i}|}{k} \left\| a^{(n)}- a^{(n-1)} \right\|_k,
		\end{align*}
		which gives us a contraction mapping from $D^\kappa$ to $D^\kappa$. Hence, by the Banach fixed point theorem, there exists a unique solution to \eqref{aN}.
	\end{proof}
	Next, we want to show that for large enough $N$, the solution to \eqref{aN} is also the solution to \eqref{eq:a}. 
	\begin{lemma}\label{l:bdN}
		For $$N \ge e^{KT} \left(\sum_{y=1}^{\kappa} g_y+T \sum_{y=1}^{\kappa} h_y \right),$$ where $K := C_1 \max_{y \in \mathcal Y} \left\lvert \tilde{b}_{1y} \right\rvert_{\infty} + \max_{i \in \mathcal Y} \sum_{y=1}^{\kappa} |q_{y,i}|$, the solution $a^{(N)}$ to \eqref{aN} satisfies the inequalities
		\begin{align}
			\label{eq:bound}
			0 \le g_y + \int_t^T \left(h_y +C_1 \tilde{b}_{1y}(s) a_y^{(N)}(s) - C_2 \tilde{b}_{2y}^2(s) \left(a_y^{(N)}(s) \right)^2 + \sum_{i=1}^{\kappa} q_{y,i} a_i^{(N)}(s) \right) ds \le N 
		\end{align}
		for all $t \in [0,T]$, where $y \in \mathcal Y$.
	\end{lemma}
	\begin{proof}
		For simplicity of notations, $a_y$ is used instead of $a_y^{(N)}$ for $y \in \mathcal Y$ if there is no confusion. 
		
		First, for $y \in \mathcal Y$, we prove the positiveness of $a_y$ by contradiction. Suppose $a_y \ (y \in \mathcal Y)$ are not positive functions on $[0,T]$. Since $a_1$ is continuous and $a_1(T) = g_1 > 0$, there exists some $\tau_1 \in [0,T]$ as the closest time to $T$ such that $a_1(\tau_1) = 0$. Note that finding such a $\tau_1$ is possible. Let $t_n \in [0,T]$ be a non-decreasing sequence such that $a_1(t_n) = 0$, there exists some $\tau_1$ such that $t_n \rightarrow \tau_1 < T$ as $n \to \infty$ since $a_1$ is continuous and $a_1(T) = g_1 >0$. By the continuity of $a_1$, we have $a_1(\tau_1) = 0$, which gives the desirable point $\tau_1$. Then for all $t \in (\tau_1,T]$, $a_1(t) > 0$ and it implies that $a_1'(\tau_1) >0$. In this case, plugging $t = \tau_1$ to \eqref{eq:a}, we have 
		$$a_1'(\tau_1) = -h_1- \sum_{i\neq 1}^{\kappa} q_{1,i} a_i(\tau_1) > 0,$$ 
		which implies there is some $y \in \mathcal Y$ and $y \neq 1$ such that $a_y(\tau_1) < 0$. Without loss of generality, we let $a_2(\tau_1) < 0$. Since $a_2$ is continuous on $[0,T]$ and $a_2(T) = g_2 > 0$, from the intermediate value theorem, there exists some $\tau_2 \in (\tau_1,T)$ such that $a_2(\tau_2) = 0$ and $a_2'(\tau_2) > 0$. This indicates that $a_2'(\tau_2) = -h_2 - \sum_{i \neq 2}^{\kappa} q_{2,i} a_i(\tau_2) > 0$
		by plugging $t = \tau_2$ back to \eqref{eq:a}, and it implies that there is some $y \in \mathcal Y$ and $y \neq 1, 2$ such that $a_y(\tau_2) < 0$ since we already know $a_1(\tau_2) > 0$. Without loss of generality, we can let $a_3(\tau_2) < 0$. By induction with the same argument, there is a $\tau_{\kappa} \in (\tau_{\kappa - 1}, T)$ such that $a_{\kappa}(\tau_{\kappa}) = 0$ and $a_{\kappa}'(\tau_{\kappa}) > 0$, which gives
		$$a_{\kappa}'(\tau_{\kappa}) + h_{\kappa} + \sum_{i \neq \kappa}^{\kappa} q_{\kappa,i} a_i(\tau_{\kappa}) = 0.$$
		But it contradicts with the fact that 
		$$a_{\kappa}'(\tau_{\kappa}) > 0, \ h_{\kappa} > 0, \ q_{\kappa, i} > 0, \ a_i(\tau_{\kappa}) > 0$$ 
		for $i \in \{1, 2, \dots, \kappa - 1\}$. Thus the positiveness of $a_y$ on $[0, T]$ for all $y \in \mathcal Y$ is obtained.
		
		Next, we prove the upper boundness for the integral in \eqref{eq:bound}. Note that for all $t \in [0, T]$ and $y \in \mathcal Y$, let $\tau = T - t$, we have
		$$a_y'(\tau) = h_y + C_1 \tilde{b}_{1y}(\tau) a_y(\tau) - C_2 \tilde{b}^2_{2y}(\tau) a_y^2(\tau) + \sum_{i=1}^{\kappa}q_{y,i}a_i(\tau),$$
		and thus
		\begin{align*}
			\sum_{y=1}^{\kappa} a_y'(\tau)
			&= \sum_{y=1}^{\kappa} h_y + C_1 \sum_{y=1}^{\kappa} \tilde{b}_{1y}(\tau) a_y(\tau) - C_2 \sum_{y=1}^{\kappa} \tilde{b}^2_{2y}(\tau) a_y^2(\tau) + \sum_{y=1}^{\kappa} \sum_{i=1}^{\kappa}q_{y,i}a_i(\tau) \\
			&\leq \sum_{y=1}^{\kappa} h_y + C_1 \max_{y \in \mathcal Y} \left\lvert \tilde{b}_{1y} \right\rvert_{\infty} \sum_{y=1}^{\kappa} a_y(\tau) + \sum_{y=1}^{\kappa} \sum_{i=1}^{\kappa} |q_{y,i}| a_i(\tau) \\
			& \leq \sum_{y=1}^{\kappa} h_y + \sum_{i=1}^{\kappa} \left( C_1 \max_{y \in \mathcal Y} \left\lvert \tilde{b}_{1y} \right\rvert_{\infty} +  \sum_{y=1}^{\kappa} |q_{y,i}| \right) a_i(\tau) \\
			& \leq \sum_{y=1}^{\kappa} h_y + K \sum_{i=1}^{\kappa} a_i(\tau),
		\end{align*}
		where 
		$$K := C_1 \max_{y \in \mathcal Y} \left\lvert \tilde{b}_{1y} \right\rvert_{\infty} + \max_{i \in \mathcal Y} \sum_{y=1}^{\kappa} |q_{y,i}|$$
		with $\sum_{y=1}^{\kappa} a_y(T) = \sum_{y=1}^{\kappa} g_y$. By Gr\"onwall's inequality, for all $\tau \in [0, T]$,
		$$\sum_{y=1}^{\kappa} a_y(\tau) \le e^{KT} \left(\sum_{y=1}^{\kappa} g_y+T \sum_{y=1}^{\kappa} h_y \right).$$ 
		Hence $a_y(t) \le e^{KT} \left(\sum_{y=1}^{\kappa} g_y+T \sum_{y=1}^{\kappa} h_y \right)$ for all $t \in [0, T]$ and $y \in \mathcal Y$.
		Hence, when 
		$$e^{KT} \left(\sum_{y=1}^{\kappa} g_y+T \sum_{y=1}^{\kappa} h_y \right) \leq N,$$
		\eqref{eq:bound} holds.
	\end{proof}

	\begin{lemma}
		\label{l:eu_abc}
		With the given of $h_y, g_y \in \mathbb R^{+}$, $y \in \mathcal Y$, there exists a unique solution to the Riccati system \eqref{eq:ode1}.
	\end{lemma}
	\begin{proof}
		The existence, uniqueness and boundedness of the solution to $a_y$ ($y \in \mathcal Y$) are shown in Lemma \ref{l:existN} and Lemma \ref{l:bdN}. Given $(a_y: y \in \mathcal Y)$, the coefficient functions $b_y$ ($y \in \mathcal Y$) form a linear ordinary differential equation system. Their existence and uniqueness are guaranteed by Theorem 12.1 in \cite{antsaklis2006}. Similarly, with the given of $(a_y, b_y: y \in \mathcal Y)$, the coefficient functions $c_y, k_y$ ($y \in \mathcal Y$) also form a linear ordinary differential equation system. Applying the Theorem 12.1 in \cite{antsaklis2006}, we can obtain the existence and uniqueness of $c_y, k_y$ ($y \in \mathcal Y$).
	\end{proof}

	\subsection{Multidimensional Problem}
	\label{s:section5}
	
	In this subsection, we consider the multidimensional problem, which is a straightforward extension of the previous one-dimensional setup. The same type of Ricatti system to characterize the equilibrium and the value function is obtained, and we have a similar result as the Theorem \ref{t:main}.
	
	Suppose that $X_t$, $W_t$ and $\alpha_t$ take values in $\mathbb R^d$, and all components of $W_t$ are independent. Suppose that the dynamic of the generic player is given by
	\begin{equation*}
		X_t = X_0 + \int_0^t \left(\tilde{b}_1(Y_s, s) X_s + \tilde{b}_2(Y_s, s) \alpha_s \right) ds + W_t.
	\end{equation*}
	Consider the cost function
	\begin{equation*}
		\begin{aligned}
			&J[m](y,x,t,\bar\mu,\bar\nu) \\
			=& \mathbb E \left[\int_t^T \left(\frac{1}{2} \|\alpha_s\|_2^2 +h(Y_s) \int_{\mathbb R^d} \|X_s-z\|_2^2 m(dz) \right)ds+ \right. \\
			& \left.\left. g(Y_T)\int_{\mathbb R^d} \|X_T-z\|_2^2 m(dz) \right\vert  X_t = x, Y_t = y, \mu_t = \bar\mu, \nu_t = \bar\nu\right]  \\
			=& \mathbb E\left[\int_t^T \left(\frac{1}{2} \alpha_s^\top \alpha_s +h(Y_s) \left(X_s^\top X_s -2\mu_s^\top X_s+ \nu_s \cdot\1_d \right)\right)ds + \right. \\
			& \left.\left. g(Y_T)\left(X_T^\top X_T -2\mu_T^\top X_T+ \nu_T \cdot\1_d \right)\right\vert  X_t = x, Y_t = y, \mu_t = \bar\mu,\nu_t = \bar\nu \right],
		\end{aligned}
	\end{equation*}
	where $m$ is the joint density function in $\mathbb R^d$, and $\mu, \nu$ take value in $\mathbb R^d$. For $y \in \mathcal Y$, define the Riccati system
	\begin{equation}
		\label{eq:ODEsys_multi}
		\begin{cases}
				\displaystyle a_y' + 2\tilde{b}_{1y} a_y - 2\tilde{b}^2_{2y} a_y^2 + \sum_{i=1}^{\kappa} q_{y,i} a_i + h_y = 0, \\
				\displaystyle b_y' + \left(2 \tilde{b}_{1y} - 4\tilde{b}^2_{2y} a_y \right) b_y + \sum_{i=1}^{\kappa} q_{y,i} b_i + h_y  =0, \\
				\displaystyle c_y' + d a_y + d b_y + \sum_{i=1}^{\kappa} q_{y,i} c_i  = 0,\\
				\displaystyle k_y' -2\tilde{b}^2_{2y} a_y^2 + 4 \tilde{b}^2_{2y} a_y b_y + 2 \tilde{b}_{1y} k_y + \sum_{i=1}^{\kappa} q_{y,i} k_i = 0, \\
				\displaystyle a_y(T) = b_y(T) = g_y\text{ , } c_y(T) =  k_y(T) = 0. 
			\end{cases}
	\end{equation}
	\begin{theorem}[Verification theorem for MFGs]
		\label{t:main_multi}
		There exists a unique solution 
		$(a_y, b_y, c_y, k_y: y \in \mathcal Y)$ for the Riccati system \eqref{eq:ODEsys_multi}. 
		With these solutions, for $t \in [0, T]$,
		the MFG equilibrium path follows $\hat X = \hat X [\hat m]$ is given by
		\begin{equation*}
			d \hat X_t = \left( \tilde{b}_1(Y_t, t) \hat{X}_t - 2 \tilde{b}_2^2(Y_t, t)  a_{Y_t}(t) \left(\hat X_t - \hat{\mu}_t \right) \right)  dt + dW_t, \quad \hat X_0 = X_0,
		\end{equation*}
		with equilibrium control $
			\hat \alpha_t = - 2 \tilde{b}_2(Y_t, t) a_{Y_t}(t) \left(\hat X_t - \hat{\mu}_t \right),$
	    where 
	    $$d \hat{\mu}_t = \tilde{b}_1(Y_t, t) \hat{\mu_t} dt, \quad \hat{\mu}_0 = \mathbb E[X_0].$$
		Moreover, the value function $U$ 
		is 
		$$U (m_0, y, x) =a_y(0) x^\top x -2a_y(0)x^\top[m_0]_1 + k_y(0)[m_0]_1^\top [m_0]_1+ b_y(0) [m_0]_2^\top \1_d + c_y(0)$$
		for $y \in \mathcal Y$.
	\end{theorem}
	The proof is similar to the one-dimensional problem, and we don't show the details here.

\section*{Acknowledgments}
We would like to acknowledge valuable discussions and insightful examples provided by Prof Jianfeng Zhang of the University of Southern California.


\bibliographystyle{plain}


\end{document}